\documentclass[11pt]{amsart}

\usepackage{amssymb, amsmath, amsthm, amsfonts, mathrsfs}
\usepackage{bbm}
\usepackage[dvipsnames]{xcolor}
\usepackage{comment}
\usepackage{mathtools}
\usepackage[normalem]{ulem}
\usepackage{hyperref}

\usepackage[margin=1.2in]{geometry}

\newtheorem{definition}{Definition}[section]
\newtheorem{theorem}[definition]{Theorem}
\newtheorem{corollary}[definition]{Corollary}
\newtheorem{proposition}[definition]{Proposition}
\newtheorem{lemma}[definition]{Lemma}
\newtheorem{remark}[definition]{Remark}
\newtheorem{assumption}{Assumption} 

\numberwithin{equation}{section}

\newcommand{\R}{\mathbb{R}}
\newcommand{\E}{\mathbb{E}}
\newcommand{\N}{\mathbb{N}}

\newcommand{\di}{\mathrm{d}}
\newcommand{\eps}{\varepsilon}

\title[SDE\MakeLowercase{s} with singular coefficients]{SDE\MakeLowercase{s} with singular coefficients: the martingale problem view and the stochastic dynamics view}

\author{Elena Issoglio* and Francesco Russo}
\address[Elena Issoglio]{Dipartimento di Matematica `G.\ Peano', Universit\'a di Torino}
\email[Corresponding author]{elena.issoglio@unito.it}

\address[Francesco Russo]{Unit\'e de Math\'{e}matiques appliqu\'{e}es, ENSTA Paris, Institut Polytechnique de Paris}
\email{francesco.russo@ensta-paris.fr}

\date{13th February 2024.\\
\phantom{xx}*corresponding author}

\begin{document}

\begin{abstract}
We consider SDEs with (distributional) drift in negative Besov spaces and random initial condition and investigate them from two different viewpoints. 
In the first part we set up a martingale problem and  show its well-posedness.
We then prove further properties of the martingale problem, like continuity with respect to the drift and the link with the Fokker-Planck equation. We also show that the solutions are weak Dirichlet processes for which we evaluate the quadratic variation of   the martingale component.
In the second part we identify the dynamics of the solution of the martingale problem
by describing the proper associated SDE.
Under suitable assumptions we show equivalence with the solution to the  martingale problem.
\end{abstract}

\maketitle


\section{Introduction}
In this paper we  study the  SDE
\begin{equation}\label{eq:SDE}
\mathrm d X_t = b(t, X_t) \mathrm dt + \mathrm d W_t, \quad X_0 \sim \mu,
\end{equation}
where $X_t\in \R^d$,  the process $(W_t)$ is a $d$-dimensional Brownian motion, $\mu$ is any probability measure and the drift $b(t, \cdot)$ is an element of a negative Besov space  $\mathcal C^{(-\beta)+}$, see below for the precise definition.
 SDE \eqref{eq:SDE} is clearly only formal at this stage, because the drift $b$ cannot even  be evaluated at the point $X_t$, and one
first needs to define a notion of solution for this kind of SDEs. 
We tackle this problem from two different viewpoints. In the first part we set up a martingale problem and  show its well-posedness. In the second part we identify the dynamics of the solution of the martingale problem.

The first steps in the study of the SDE in dimension 1
(and with a diffusion coefficient $\sigma$) were done in \cite{frw1, frw2,russo_trutnau07}.
In dimension $d>1$ we mention the work \cite{flandoli_et.al14} where the authors introduced the notion of
\emph{virtual solution}   whose construction  depended a priori on a real parameter  $\lambda$. Also,
the setting was slightly
different because the function spaces were negative fractional  Sobolev spaces $H^{-\beta}_q$
and not Besov spaces. Other authors have studied SDEs with distributional coefficients afterwards, we mention in particular  \cite{diel, cannizzaro, ZhangZhao, athreya2020}. The main idea in all these works, which is the same we also develop in the first part of the present paper, is to frame the SDE as a martingale problem, hence the main goal is to  find  a domain $\mathcal D_{\mathcal L}$ that characterises the martingale solution in terms of the quantity
\begin{equation}\label{eq:mart}
f(t, X_t) -f(0, X_0) - \int_0^t \mathcal Lf (s, X_s) \di s,
\end{equation}
for all $f\in \mathcal D_{\mathcal L}$, where $\mathcal L$ is the {parabolic generator} of $X$ formally given by  $\mathcal L f =\partial_t  f + \tfrac12 \Delta f  +\nabla f \, b$. This is made rigorous  using  results on the PDE 
\begin{equation*}
\left\{
\begin{array}{l}
\mathcal L f  = g\\
 f(T) = f_T,
 \end{array} \right.
\end{equation*}
developed in \cite{issoglio_russoPDEa}.

Our framework in terms of function spaces is slightly different than all the works cited above. In the first part of the article,
the only difference is that we allow the initial condition  $X_0$ to be any random variable, and not only a Dirac delta in a point $x$.
Well-posedeness of the PDE $\mathcal Lf =g$ allows to give a proper meaning to the martingale problem. Various regularity results on the PDE together with a transformation of the solution $X$ into the solution $Y$ of a `standard'
(Stroock-Varadhan) martingale problem (see Section \ref{sc:zvonkin}), allow us to show existence and uniqueness of the solution $X$ to the martingale problem, see Theorem \ref{thm:X0}.
 We also prove other interesting results such as Theorem \ref{thm:McKvFP} where we show that the law density of the solution $X$ satisfies the Fokker-Planck equation, which is  a PDE with negative Besov coefficients. Furthermore we  show in Theorem \ref{thm:tight} some tightness  results for smoothed  solutions  $X^n$ when the negative Besov coefficients are smoothed. 
 
 \vspace{5pt}
The main novelty of this paper is the second part, where
we study the SDE $ X_t = X_0 + \int_0^t b(s, X_s) \di s + W_t $ from a different point of view, in particular we look into the dynamics of the process itself. One natural question to ask, which is well understood in the classical Stroock-Varadhan case where $b$ is a locally bounded function, is the equivalence between the solution to the martingale problem and the solution in law (i.e.\ weak solution) of the SDE. In the case of SDEs with distributional coefficients, the first challenging problem is to define a suitable notion of solution of the SDE, and then to study well-posedness of that equation. 
To this aim, we start in
 Section \ref{sc:weakDir} by showing that the solution to the martingale problem is a weak Dirichlet process, for which we identify the martingale component in its canonical decomposition, see Proposition \ref{pr:Mf} and Remark \ref{rm:A}.
  We then introduce in Section  \ref{sc:generalisedSDE} our notion of solution for the  SDE, involving a `local time' operator which plays the role of the integral $\int_0^t b(s, X_s) \di s$ and involving weak Dirichlet processes.
  Under further mild assumptions on $b$, for example if it has compact support, in Theorem \ref{thm:cr} we show that a solution to the martingale problem is also a solution to the SDE.
  In a  slightly more restricted framework,   in Proposition \ref{pr:Xfqv} we obtain the converse result, hence providing the equivalence result of SDEs and martingale problems for distributional drifts, see Corollary \ref{cor:Xsoliff}. Those results extend
\cite[Propositions 6.7 and 6.10]{russo_trutnau07} stated in dimension 1 and in
  the case of time-homogeneous coefficients.
  
  A typical example of drift $b$ for which all our results are valid, arises when $b$ is a quenched realization of  an independent noise $\dot B_x(\omega)$, which is a generalised random field whose trajectories are the divergence of
  a $(1-\beta)$-H\"older continuous functions $x\mapsto B_x(\omega)$ for some $\beta\in(0,\frac12)$,
  cut with a smooth function with compact support. These models arise when describing the motion of particles propagating in an irregular medium, see \cite{seignourel1998processus} and references therein. 
  The class of these noises
  is large and  in dimension $d=1$ it includes for instance (bi)fractional,
  multi-fractional Brownian ones, etc, with Hurst index greater than $\frac12$, to be cut so that they have compact support.

 A result connected to ours is provided by \cite{chaudru_menozzi}, where the authors study the case when the driving noise is a
 L\'evy $\alpha$-stable process and the distributional drift lives in a general Besov space $\mathbb B^{-\beta}_{p,q}$.
 In particular they formulate the martingale problem and  a quite different notion of SDE (for which, in  $d=1$ they even study pathwise uniqueness, extending in this way \cite[Corollary 5.19]{russo_trutnau07}, stated for Brownian motion) and prove that a solution to the martingale problem is also a solution to their SDE. However, they do not prove the converse result, hence they do not have any equivalence.

\vspace{5pt}

The paper is organised as follows. In Section \ref{sc:prelim} we introduce the framework in which we work, in particular the various functions spaces appearing in the paper and many useful results from the companion paper \cite{issoglio_russoPDEa}.
In Section \ref{sc:zvonkin} we introduce  the  martingale problem and transform it into a classical equivalent Stroock-Varadhan martingale problem. 
In Section \ref{sc:MP} we show existence and uniqueness of a solution to the martingale problem and various other properties. 
In Section \ref{sc:weakDir} we show that the solution to the martingale problem is a weak Dirichlet process and identify its decomposition. 
In Section  \ref{sc:generalisedSDE} we introduce the notion of solution to the SDE and show its equivalence to the martingale problem. 
Finally, in Appendix \ref{app:lunardi} we state a useful result on solutions of (classical) PDEs that we use in the paper.

\section{Setting and preliminary results}\label{sc:prelim} 

\subsection{Function spaces}
Let us denote by $ C^{1,2}_{buc}:= C^{1,2}_{buc}([0,T]\times \R^d)$  the space of all $C^{1,2}$ real functions such that the function and its gradient in $x$ are bounded, and the Hessian matrix and the time-derivative are bounded and uniformly continuous. Let  us denote by $ C^{1,2}_{c}:= C^{1,2}_{c}([0,T]\times \R^d)$ the space of  $   C^{1,2} ([0,T]\times \R^d)$ with compact support.
Let us denote by $C_b^{1,2}:= C^{1,2}_{b}([0,T]\times \R^d) $ the space of $C^{1,2}$-functions that are bounded with bounded derivatives. 
We also use the notation $C^{0,1}:=C^{0,1}([0,T]\times \R^d)$ to indicate the space of real functions with gradient in $x$ uniformly continuous in $(t,x)$.  
Let $C_c^\infty:= C_c^\infty(\R^d)$ denote  the space of all smooth real functions with compact support. 
We denote by $C_c=C_c(\R^d)$ the space of $\R$-valued continuous functions with compact support.
Let $\mathcal S=\mathcal S(\mathbb R^d )$ be the space of real-valued Schwartz functions on $\mathbb R^d$ and $\mathcal S'=\mathcal S'(\mathbb R^d )$ the space of Schwartz distributions. The corresponding dual pairing will be denoted by $\langle \cdot, \cdot \rangle$.

 For $\gamma\in\mathbb R$ we denote by  $\mathcal C^\gamma = \mathcal C^\gamma(\mathbb R^d)$ the Besov space (or H\"older-Zygmund space), endowed with its norm $\|\cdot\|_{\gamma}$.  For more details see 
\cite[Section 2.7, pag 99]{bahouri} and also \cite{issoglio_russoPDEa}, where we recall all useful facts and definitions about these spaces. 
If $\gamma \in \R^+ \setminus \mathbb N$ then the space coincides with the classical H\"older space. If $\gamma<0$ then the space includes some Schwartz distributions. We have $\mathcal C^\gamma \subset \mathcal C^\alpha$ for any $\gamma >\alpha$. Moreover it holds that $L^\infty \subset \mathcal C^0$ (see \cite{iprt24} for a proof in the case of anisotropic Besov spaces). We denote by $C_T \mathcal C^\gamma$ the space of continuous functions on $[0,T]$ taking values in $\mathcal C^\gamma$, that is $C_T \mathcal C^\gamma:= C([0,T]; \mathcal C^\gamma)$.
For any given $\gamma\in \R$ we denote by $\mathcal C^{\gamma+}$ and $\mathcal C^{\gamma-}$  the spaces given by
\[
\mathcal C^{\gamma+}:= \cup_{\alpha >\gamma} \mathcal C^{\alpha} ,  \qquad  
\mathcal C^{\gamma-}:= \cap_{\alpha <\gamma} \mathcal C^{\alpha}.
\]
Note that $\mathcal C^{\gamma+}$ is an inductive space. 
We will also use the spaces $C_T C^{\gamma+}:=C([0,T]; \mathcal C^{\gamma+})$, which is equivalent to the fact that for $f\in C_T C^{\gamma+} $ there exists $\alpha>\gamma $ such that $f\in C_T C^{\alpha}$,  see for example \cite[Appendix B]{issoglio_russoMK}. 
Similarly, we use the space   $C_T C^{\gamma-}:=C([0,T]; \mathcal C^{\gamma-})$, meaning that if $f\in C_T \mathcal C^{\gamma-} $ then for any $\alpha<\gamma $ we have $f\in C_T \mathcal C^{\alpha}$.  
 We denote by $\mathcal C_c^{\gamma}=\mathcal C_c^{\gamma}(\R^d)$ the space of elements in $\mathcal C^{\gamma}$ with compact support. Similarly when $\gamma$ is replaced by $\gamma+$ or $\gamma-$. When defining the domain of the martingale problem  we will work with spaces of functions which are the limit of functions with compact support, so that they are Banach spaces. More precisely, let us denote by $\bar{\mathcal C}_c^\gamma  = \bar{\mathcal C}_c^\gamma (\R^d)$ the space  
\[
\bar{\mathcal C}_c^\gamma := \{f \in {\mathcal C}^\gamma \text{ such that } \exists (f_n) \subset {\mathcal C}_c^\gamma \text{ and } f_n \to f \text{ in }  {\mathcal C}^\gamma\}.
\]
As above we denote the inductive space and intersection space as
\[
\bar {\mathcal C}_c^{\gamma+}:= \cup_{\alpha >\gamma} \bar{\mathcal C}_c^{\alpha} ,  \qquad  
\bar {\mathcal C}_c^{\gamma-}:= \cap_{\alpha <\gamma} \bar{\mathcal C}_c^{\alpha}.
\]
The main reason for introducing this class of subspaces is that  $\bar {\mathcal C}_c^{\gamma+}$ are separable, as proved in  \cite[Lemma 5.7]{issoglio_russoPDEa}, unlike the classical Besov spaces $ C^{\gamma}$ and $ {\mathcal C}^{\gamma+}$ which are not separable. 
Similarly as above, we use the space   $C_T \bar{\mathcal C}_c^{\gamma+}:=C([0,T]; \bar{\mathcal C}_c^{\gamma+})$; in particular we observe that  if $f\in C_T \bar{\mathcal C}_c^{\gamma+} $ then for any $\alpha<\gamma $ we have $f\in C_T \bar{\mathcal C}_c^{\alpha}$ by \cite[Remark B.1, part (ii)]{issoglio_russoMK}.
Moreover in \cite[Corollary 5.8]{issoglio_russoPDEa} we show that  $C_T \bar{\mathcal C}_c^{\gamma+}$ is separable. 
Note that if $f$ is continuous and such that $\nabla f \in C_T \mathcal C^{0+}$ then $f\in C^{0,1}$.

Note that for all function spaces introduced above we use the same notation  to indicate  $\R$-valued functions but also $\R^d$- or $\R^{d\times d}$-valued  functions. It will be clear from the context which space is needed.  When $f:\R^d \to \R^m$ is differentiable, we denote by $\nabla f$ the matrix given by $(\nabla f)_{i,j} = \partial_i f_j$. In particular  when $f: \R^d \to \R$ then $\nabla f$ is a column vector and we denote the Hessian matrix of $f$ by Hess$(f)$.

For $\gamma\in(0,1)$  we define space $D\mathcal C^\gamma $ as
\begin{align*}\label{eq:S}
  D\mathcal C^\gamma:= \{ & h:  \R^d \to \R \text{ differentiable function s.t. } \nabla h \in \mathcal C^\gamma\},
\end{align*}
and by $C_T   D\mathcal C^\gamma:=  C([0,T]; D\mathcal C^\gamma)$.  Note that the following inclusion holds $\mathcal C^{1+\alpha}  \subset D\mathcal C^\alpha.$ 
Analogously as for the $\mathcal C^{\gamma+}$-spaces, for $\gamma>0$ we also introduce the spaces 
$$D \mathcal C^{\gamma+}:= \cup_{\alpha >\gamma} D\mathcal C^{\alpha} ,  \qquad  
D\mathcal C^{\gamma-}:= \cap_{\alpha <\gamma} D\mathcal C^{\alpha}.
$$
We will also use the spaces $C_T D\mathcal  C^{\gamma+}:=C([0,T]; D \mathcal C^{\gamma+})$. For more details on these spaces, see \cite[Section 3]{issoglio_russoPDEa}.


\subsection{Some tools and properties}\label{ssc:pointwise}
The following is an important estimate which allows to define the pointwise product between certain distributions and functions, which is based on Bony's estimates. For details see \cite{bony} or \cite[Section 2.1]{gubinelli_imkeller_perkowski}. Let   $f \in \mathcal C^\alpha$ and $g\in\mathcal C^{-\beta}$ with $\alpha-\beta>0$ and $\alpha,\beta>0$. Then the {`pointwise product'} $ f \, g$ is well-defined as an element of $\mathcal C^{-\beta}$ and  there exists a constant $c>0$ such that 
\begin{equation}\label{eq:bony}
\| f \, g\|_{-\beta} \leq c \| f \|_\alpha \|g\|_{-\beta}.
\end{equation} 
\begin{remark}\label{rm:bonyt}
Using \eqref{eq:bony} it is not difficult to see that  if $f \in C_T \mathcal C^{\alpha}$ and $g \in C_T \mathcal C^{-\beta}$ for $\alpha>\beta>0$ then the product is also continuous with values in $ \mathcal C^{-\beta}$, and 
\begin{equation}\label{eq:bonyt}
\| f \, g\|_{C_T \mathcal C^{-\beta}} \leq c \| f \|_{C_T \mathcal C^{\alpha}} \|g\|_{ C_T \mathcal C^{-\beta}}.
\end{equation} 
\end{remark}

Below we recall some results on a class of PDEs with distributional drift in negative Besov spaces that will be used to set up the martingale problem for the singular SDE \eqref{eq:SDE}. All results are taken from \cite{issoglio_russoPDEa}. In \cite{issoglio_russoPDEa}, as well as in the present work, the main assumption concerning the distribution-valued function $b$ is the following.
\begin{assumption}\label{ass:param-b}
Let $0<\beta<1/2$ and  $b\in C_T \mathcal C^{(-\beta)+}(\R^d)$. In particular $b\in C_T \mathcal C^{-\beta}(\R^d)$. Notice that $b$ is a column vector.
\end{assumption}
We start by the formal definition of the operator $\mathcal L$.
\begin{definition}[Definition 4.3, \cite{issoglio_russoPDEa}]\label{def:L}
Let $b$ satisfy Assumption \ref{ass:param-b}. The operator $\mathcal L$ is defined as  
\begin{equation*}
\begin{array}{lcll}
\mathcal L  :  &\mathcal D_{\mathcal L}^0 &\to &\{\mathcal S'\text{-valued continuous functions}\}\\
& f & \mapsto & \mathcal L f:=   \dot f + \frac12\Delta f + \nabla f \, b,
\end{array}
\end{equation*}
where 
$$\mathcal D_{\mathcal L}^0 : = C_T  D\mathcal C^{\beta} \cap C^1([0,T]; \mathcal S').$$
Here $f: [0,T]\times  \R^d \to \R$ and the function $\dot f:[0,T]\to \mathcal S'$ is the time-derivative of $f$. Note also that $\nabla f \, b := \nabla f \cdot b$ is well-defined  using  \eqref{eq:bony}  and Assumption \ref{ass:param-b} and moreover it is continuous. The Laplacian $\Delta$ is intended in the
sense of distributions.
\end{definition}

Next we recall some results  on certain PDEs, all driven by the operator $\mathcal L$. These results are all proved in the companion paper \cite{issoglio_russoPDEa}.  There are three equations
of interest, all related but slightly different. 
The first PDE is 
\begin{equation}\label{eq:PDE}
\left\{ 
\begin{array}{l}
\mathcal L v =  g
\\
v(T) = v_T.
\end{array}\right. 
\end{equation} 
We know from \cite[Remark 4.8]{issoglio_russoPDEa} that if  $v_T \in \mathcal C^{(1+ \beta)+}$ and $g \in C_T \mathcal C^{(-\beta)+}$ then there exists a unique (weak or mild) solution $v\in C_T \mathcal C^{(1+ \beta)+}$. 
In \cite[Lemma 4.17 and Remark 4.18]{issoglio_russoPDEa} we prove a  continuity result, namely that if the terminal condition $v_T$ in \eqref{eq:PDE} is replaced by a sequence $(v_T^n)$ that converges to $v_T$ in $\mathcal C^{(1+\beta)+}$, the terms $b$ and $g$ are replaced by two sequences $(b^n)$ and $(g^n)$ respectively, both converging in $ C_T \mathcal C^{-\beta}$,  then also the corresponding unique solutions $(v^n)$ will converge to $v$ in    $C_T\mathcal C^{(1+\beta)+}$.

We can solve PDE \eqref{eq:PDE} also under weaker conditions on $v_T$, in particular we allow functions with linear growth. The space that characterises this behaviour is denoted by $D\mathcal C^{\beta}$, which  is the space of differentiable functions whose gradient belongs to $\mathcal C^{\beta}$. 
Notice that in \cite{issoglio_russoPDEa} we introduce two concepts of
solution,  weak and mild, which  are defined  for   functions in $ C_T D\mathcal C^{\beta}$. We prove in \cite[Proposition 4.5]{issoglio_russoPDEa} that the notions of weak and mild solution  of the PDE are equivalent. 
 In \cite[Remark 4.8]{issoglio_russoPDEa} we show that if $v_T \in D\mathcal C^{\beta+}$ then there exists a unique   solution $v\in C_T D \mathcal C^{\beta+}$.  
Continuity results for PDE \eqref{eq:PDE} in the spaces $D\mathcal C^{\beta+}$ also hold, as we prove in  \cite[Remark 4.18 (i)]{issoglio_russoPDEa}, that is if $g^n\to g$ in $C_T \mathcal C^{-\beta} $,   $b^n \to b$ in $C_T \mathcal C^{-\beta} $ and $v^n_T \to v_T$ in $ D \mathcal C^{\beta+} $    then $v^n \to v$ in $C_T D \mathcal C^{\beta+} $. As a special case we show in \cite[Corollary 4.10]{issoglio_russoPDEa} that the function $\text{id}_i(x)= x_i$ solves PDE \eqref{eq:PDE} with $v(T)= x_i$ and $ g= b_i$, that is $\mathcal L \text{id}_i = b_i$.

Let $\lambda>0$. The second PDE to consider is
\begin{equation}\label{eq:PDEphi}
\left\{
\begin{array}{l}
\mathcal L \phi_i = \lambda (\phi_i - \text{id}_i)\\
\phi_i(T) = \text{id}_i,
\end{array}
\right.
\end{equation} 
which has a unique (weak or mild) solution $\phi_i$ for $i=1, \ldots, d$ in the space $C_T D\mathcal C^{(1-\beta)-}$ (uniqueness holds in $C_T D\mathcal C^{\beta}$) by \cite[Theorem 4.7 (i)]{issoglio_russoPDEa}.
In  \cite[Proposition 4.15]{issoglio_russoPDEa} we show that $\phi_i \in \mathcal D^0_{\mathcal L}$ and  $\dot \phi_i \in C_T \mathcal C^{(-\beta)-} $ for all  $i=1, \ldots, d$. We denote by $\phi$ the column vector with components $\phi_i$, $i=1, \ldots, d$.
We show in  \cite[Proposition 4.16]{issoglio_russoPDEa} that there exists $\lambda >0 $ large enough such that $\phi(t, \cdot)$ is invertible for all $t\in[0,T]$, and denoting such inverse with
\begin{equation}\label{eq:psi}
 \psi(t, \cdot):= \phi^{-1}(t, \cdot).
 \end{equation}
In the same proposition we also show that  $\phi, \psi \in C^{0,1}$ and moreover that  $\nabla \phi \in C_T \mathcal C^{(1-\beta)-}$ and $\nabla \psi (t, \cdot) \in \mathcal C^{(1-\beta)-}$ for all $t\in[0,T]$ and $\sup_{t\in[0,T]}\|\nabla \psi (t, \cdot) \|_{\alpha}$ for all $\alpha<1-\beta$.
From now on, let $(b^n)$ be the sequence defined in \cite[Proposition 2.4]{issoglio_russoMK}, so we know that $b^n\to b$ in $C_T\mathcal C^{-\beta}$, $b^n \in C_T \mathcal C^\gamma$ for all $\gamma>0$ and $b^n$ is bounded and Lipschitz.  Here and in the rest of the paper $\lambda>0$ is fixed and independent of $n$, chosen such that  
\begin{equation} \label{eq:lambda}
\lambda= [C(\beta, \varepsilon) \max\{ \sup_n\|b^n\|_{C_T \mathcal C^{-\beta+\eps}},\|b\|_{C_T \mathcal C^{-\beta+\eps}} \}]^{\frac1{1-\theta}},
\end{equation}
according to \cite[Lemma 4.19]{issoglio_russoPDEa}, where $\eps>0$ is such that  $\theta:= \frac{1+2\beta-\eps}{2}<1$ and $ C(\beta, \varepsilon)$ is a constant only depending on $\beta$ and $ \varepsilon$. Notice with this choice of $\lambda$ the corresponding inverse  $\psi^n$ of $\phi^n $, see \eqref{eq:psi} is well-defined  according to \cite[Proposition 4.16 (ii)]{issoglio_russoPDEa}.
In \cite[Lemma 4.19]{issoglio_russoPDEa} we show that  $\phi^n \to \phi$ and $\psi^n \to \psi$ uniformly on $[0,T]\times \R^d$  and $\|\nabla \phi^n \|_\infty + |\phi^n(0,0)|$ is uniformly bounded in $n$.

Finally in \cite[Theorem 4.14]{issoglio_russoPDEa} we show that the function $\phi $ is equivalently defined as $\phi = \text{id} + u$, where $u= (u_1, \ldots, u_d)$ and $u_i$ is the unique solution of the third PDE, that is
\begin{equation}\label{eq:PDEu}
\left\{
\begin{array}{l}
 \mathcal L  u_i = \lambda u_i - b_i\\
u_i(T) = 0
\end{array}
\right.
\end{equation}
in the space $C_T \mathcal C^{(2-\beta)-}$.  For the latter PDE there are also continuity results proven in \cite[Lemma 4.17]{issoglio_russoPDEa}, namely  $u_i^n \to u_i$ in $C_T \mathcal C^{(2-\beta)-}$. Moreover we have uniform convergence of $u^n\to u, \nabla u^n \to \nabla u$ by \cite[Lemma 4.19]{issoglio_russoPDEa}. With $\lambda $ chosen as in \eqref{eq:lambda} we have $\| \nabla u^n\|_{\infty} \leq \frac12$ by \cite[Proposition 4.13 and bound (4.34)]{issoglio_russoPDEa}. 

\subsection{Probabilistic notation}\label{ssc:prob-notation}
In the sequel we will consider generic measurable spaces $(\Omega, \mathcal F)$. On them we will consider various probability measures denoted by $\mathbb P$. We will make use of the notation $(X, \mathbb P)$ or $(Y, \mathbb P)$, where $X$ or $Y$ will denote continuous stochastic processes indexed by $t\in [0,T]$  defined on the probability space $ (\Omega, \mathcal F, \mathbb P)$, without recalling it explicitly. The filtrations considered, if not explicitly mentioned,  will  be the canonical filtrations generated by $X$ or $Y$ (which will be the same in our applications). 

Once the probability space  $ (\Omega, \mathcal F, \mathbb P)$ is fixed, we will denote by $\mathscr C$ the linear space of continuous processes on $[0,T]$ with values in $\R^d$ endowed with the metric of uniform convergence in probability (u.c.p.).

The canonical space of continuous functions from $[0,T]$ with values in $\mathbb R^d$ will be denoted by $\mathcal C_T $, and it will be endowed with  the sigma-algebra of Borel sets  $\mathcal B(\mathcal C_T)$. For $s\in [0,T]$ we will use the notation  $\mathcal C_s $ for the space of continuous functions defined on $[0,s]$. Thus, for a given couple $(X, \mathbb P)$, the law of $X$ under $\mathbb P$ will be a Borel probability measure on the measurable space $(\mathcal C_T, \mathcal B(\mathcal C_T))$.


 \section{A Zvonkin-type transformation}\label{sc:zvonkin}

 In the study of  SDEs with low-regularity coefficients, like \eqref{eq:SDE}, one successful idea is to apply a bijective transformation that  changes the singular drift and produces a transformed SDE whose drift has no singular component and which can thus be  solved with standard techniques.  The idea goes back to Zvonkin \cite{z},
 and in the present case a transformation that does the job is the unique solution $\phi$ of the PDE \eqref{eq:PDEphi}. The analysis that we do here can shed some light on what kind of   transformations, aside from $\phi$, of the martingale problem fulfilled by $X$ will lead to different, but equivalent, transformed martingale problems   fulfilled by a new process $Y$.

 \vspace{10pt}

Let us start by introducing  a class of function, denoted by $\mathcal D_{\mathcal L} $, that is the domain of the martingale problem 

\begin{equation}\label{eq:D}
\begin{array}{ll}
\mathcal D_{\mathcal L} : = &  \{ f \in C_T\mathcal C^{(1+\beta)+}:   \exists g \in  C_T\bar {\mathcal C}_c^{0+}  \text{ such that } \\
                            &   f \text{ is a weak solution of }
                              \mathcal  L f =g  \text{ and } f(T) \in \bar{\mathcal C}_c^{(1+\beta)+  }\},
\end{array} 
\end{equation}
where  $\mathcal L$ has been defined in Definition \ref{def:L}.

\begin{definition}\label{def:MP}
We say that a couple  $(X, \mathbb P)$ where $X$ is a continuous process indexed by $t\in[0,T]$ and $\mathbb P$ is a probability on some measurable space,  is a {\em solution to the martingale problem with distributional drift $b$ and initial condition $\mu$} (for shortness, solution of MP with distributional drift $b$ and i.c.\ $\mu$) if and only if for every $f \in \mathcal D_{\mathcal L}$
\begin{equation}\label{eq:MP}
f(t, X_t) - f(0, X_0) - \int_0^t (\mathcal L f) (s, X_s) \di s
\end{equation}
is a local martingale under $\mathbb P$, and $X_0 \sim \mu$ under $\mathbb P$, where the domain  $\mathcal D_{\mathcal L} $ is given by
\eqref{eq:D} 
and  $\mathcal L$ has been defined in  Definition \ref{def:L}.

We say that the martingale problem with distributional drift $b$ {\em admits uniqueness} if for any two solutions $(X^1, \mathbb P^1)$ and $(X^2, \mathbb P^2)$ with $X^i_0\sim \mu$, $i=1,2$, then the law of $X^1$ under $\mathbb P^1$  is the same as the law of $X^2$ under $\mathbb P^2$.
\end{definition}

 \begin{remark}\label{rm:D}
Since $ \bar{\mathcal C}_c^{(1+\beta)+  } \subset \bar{\mathcal C}_c^{0+  } \subset {\mathcal C}^{(-\beta)+  }$, then there exists a unique weak solution  $f\in C_T\mathcal C^{(1+\beta)+}$ for the PDE appearing in $\mathcal D_{\mathcal L}$, see Section \ref{ssc:pointwise}.
Moreover by \cite[Remark 4.4]{issoglio_russoPDEa} we have $ \mathcal D_{\mathcal L} \subset \mathcal D_{\mathcal L}^0$.
\end{remark}

\begin{proposition}\label{pr:DLsep}
The domain $\mathcal D_{\mathcal L}$ defined in \eqref{eq:D} equipped with its graph topology is separable. 
\end{proposition}
\begin{proof}
  By  \cite[Lemma 5.7 (i)]{issoglio_russoPDEa}
  with $\gamma=0$ we know that $ \bar{\mathcal C}_c^{0+}$ is separable, hence   there exists a dense subset $D_0$ of $ \bar{\mathcal C}_c^{0+}$, and by  \cite[Corollary 5.8]{issoglio_russoPDEa} we know that $ C_T\bar{\mathcal C}_c^{\beta+}$ is separable, thus there  exists a dense subset $D_\beta$ of $ C_T\bar{\mathcal C}_c^{\beta+}$. 
Let us denote by  $D$ the set of all $ f_n\in C_T \mathcal C^{(1+\beta)+} $ such that $\mathcal L f_n =g_n; f_n(T) = f_n^T$ where $g_n \in D_0 $ and $f_n^T \in D_\beta$. 
Clearly $D$ is countable, because $D_0$ and $D_{\beta}$ are countable  and $D\subset \mathcal D_{\mathcal L}$. Moreover  by continuity results on the PDE \eqref{eq:PDE}, see Section \ref{ssc:pointwise}, we have that if $f_n^T\to f(T)$ in $\mathcal C^{(1+\beta)+}$ and $g_n \to g$ in $C_T\mathcal C^{0+}$, then $f_n \to f$ in $C_T\mathcal C^{(1+\beta)+}$, which proves that  the set $D$ is dense in $\mathcal D_{\mathcal L}$.
\end{proof}

Next we introduce the transformed  SDE studied here, which is
\begin{align}\label{eq:SDEY}
Y_t=& Y_0 + \lambda \int_0^t Y_s \di s - \lambda\int_0^t \psi(s, Y_s) \di s + \int_0^t \nabla\phi(s, \psi(s, Y_s) )\di W_s,
\end{align}
where $\phi$ is the unique solution  of \eqref{eq:PDEphi} and $\psi$ is its (space-)inverse given by \eqref{eq:psi}
with $\lambda>0$  chosen large enough (see Section \ref{ssc:pointwise}). 
Notice that this SDE is formally obtained by applying the transformation $\phi$ to $X$ as in Definition \ref{def:MP}, that is, setting $Y_t =\phi(t,X_t) $ and using that $\phi$ is invertible with inverse $\psi$.

Denoting by $Y$ the solution of \eqref{eq:SDEY}, by It\^o's formula for all $\tilde f\in C_{buc}^{1,2}([0,T]\times \R^d)$  we know that 
\[
\tilde f( t, Y_t) - \tilde f(0, Y_0) - \int_0^t ( \tilde{\mathcal L}\tilde f) (s, Y_s) \di s
\]
is a martingale under $\mathbb P$. Here the operator $\tilde {\mathcal L }$ is the generator of $Y$, which is  defined by
\begin{align}\label{eq:Ltilde}
\tilde{\mathcal L} \tilde f := \partial_t \tilde f + \lambda \nabla \tilde f (\text{id} - \psi )  +  \frac12 \text{Tr}[(\nabla \phi \circ \psi)^\top  \text{Hess} \tilde f (\nabla \phi \circ \psi )] .
\end{align}
In particular,  $(Y, \mathbb P)$ verifies the classical Stroock-Varadhan martingale problem with respect to $\tilde {\mathcal L}$.
We recall that this notion is equivalent
to the one of weak solution for SDEs, see \cite[Proposition 4.11 in Chapter 5]{karatzasShreve}. To avoid confusion with the notion of weak solution for PDEs, in this paper we use the terminology \emph{solution in law} instead of weak solution when referring to SDEs.

\begin{remark}\label{rm:tildeL}
Note that the coefficients in $\tilde{\mathcal  L}$ belong to $ C^{0,\nu}$ for any $\nu<1-\beta$, see Appendix \ref{app:lunardi} for the definition of $ C^{0,\nu}$.   Indeed $\tilde f\in C^{1,2}_{buc}$, $\psi $ has linear growth since $|\nabla \psi|$ is uniformly bounded and  the  coefficient is $\nabla \phi \circ \psi$ belongs to $ C^{0,\nu}$ for any $\nu<1-\beta$ because 
\begin{align*}
\|\nabla \phi (t, \psi(t, \cdot))\|_\nu \leq & \sup_{x} | \nabla \phi (t, \psi(t, x))|
 +\sup_{  x_1, x_2, \, x_1\neq x_2 }  \frac{| \nabla \phi (t, \psi(t, x_2)) -\nabla \phi (t, \psi(t, x_2))| }{|x_1-x_2|^\nu}\\
\leq & \sup_{t\in[0,T]}  \|\nabla \phi(t, \cdot)\|_\infty\\
 &+ \sup_{t\in[0,T]}  \sup_{  x_1, x_2, \, x_1\neq x_2 }  \frac{| \nabla \phi (t, \psi(t, x_2)) -\nabla \phi (t, \psi(t, x_2))| }{|\psi(t, x_1) - \psi(t, x_2) |^\nu} \frac{|\psi(t, x_1) - \psi(t, x_2) |^\nu }{|x_1-x_2|^\nu} \\
\leq&  \sup_{t\in[0,T]}  \|\nabla \phi(t, \cdot)\|_\infty +   \|\nabla \phi\|_{C_T \mathcal C^\nu} \|\nabla \psi \|_\infty.
\end{align*}
Here we have also used Remark \ref{rm:x1x2}.
\end{remark}

It will be useful later on to consider a  domain for the operator $\tilde{\mathcal L}$ obtained as the image of ${\mathcal D}_{\mathcal L}$ through $\phi$. Let us define 
\begin{equation}
\tilde {\mathcal D}_{\tilde {\mathcal L}} := \{\tilde f = f\circ \psi \text{ for some } f \in \mathcal D_{\mathcal L} \text{ and } \psi \text{ defined in  \eqref{eq:psi}}
\}.\label{eq:Dtilde}
\end{equation}
The choice of the SDE \eqref{eq:SDEY} and of the  domain $\tilde {\mathcal D}_{\tilde{\mathcal L}}$ are natural since we use  the transformed process $Y_t = \phi(t, X_t)$.

\begin{lemma}\label{lm:gh}
Let $g,h: \R^d \to \R^d$   with  $h\in C^1$ with $\nabla h \in C^{\beta+}$ and  $g\in  \mathcal C^{(1+\beta)+}$. Then $ g \circ h \in \mathcal C^{(1+\beta)+}$. If moreover $g_n\to g$ in $  \mathcal C^{(1+\beta)+}$ then $g_n \circ h \to g\circ h $ in $\mathcal C^{(1+\beta)+}$.
 \end{lemma}
 \begin{proof}
To prove that  $ g \circ h \in \mathcal C^{(1+\beta)+}$ is equivalent to prove that $ \bar f:= g(h(\cdot))$ is bounded and that there exists $\alpha>\beta$ such that $\nabla \bar f \in \mathcal C^\alpha$, i.e.\ $\nabla \bar f $ is bounded and $\alpha$-H\"older continuous.
 The first claim is obvious by boundedness of $ g $. The gradient $\nabla \bar f ( \cdot)  = \nabla h(\cdot) \nabla  g (h(\cdot))$ is bounded because it is the product of two bounded matrices since $\nabla h, \nabla g$ are bounded by assumption on $g, h$. 
To show that $\nabla \bar f$ is $\alpha$-H\"older continuous it is enough to show that it is the product of two functions in $\mathcal C^\alpha$ (note that boundedness of the factors is crucially used). We have $\nabla h\in \mathcal C^\alpha$ for some $\alpha>\beta$ by assumption. On the other hand it is immediate to  show that the term $\nabla  g (h( \cdot)) $ is in $\mathcal C^\alpha$, because it is bounded, and  $\alpha$-H\"older continuity is proved using that $\nabla g\in \mathcal C^\alpha$ and $h$ is Lipschitz because  by assumption $\nabla h$ is bounded.

To show convergence, let us denote $\bar f_n := g_n \circ h$. Since $\bar f_n(0)\to \bar f(0)$, it is enough to show the convergence of $\nabla \bar f_n $ in $\mathcal C^{\alpha}$. We use the same properties as above to get
$
\|\nabla \bar f_n - \nabla \bar f\|_{\alpha}  
 \leq  \|\nabla h\|_\infty  \|g_n -g\|_\alpha +   \|\nabla g_n -\nabla g\|_\infty \|h \|_\alpha,
$
and the proof is complete.
 \end{proof}

\begin{lemma}\label{lm:tildef} 
If $\tilde f\in C^{1,2}_{buc}  $  and $\phi $ is the unique solution to PDE \eqref{eq:PDEphi} then $ \tilde f \circ \phi \in C_T \mathcal C^{(1+\beta)+}$.
\end{lemma}
\begin{proof} Let us set $ f :=  \tilde f \circ \phi $. 
We first  prove that $f(t) \in \mathcal C^{(1+\beta)+}$ for all $t\in[0,T]$. This is a consequence of Lemma \ref{lm:gh} with $g= \tilde f (t,\cdot)$  and $h=\psi (t,\cdot)$. The hypothesis on $g$ are satisfied since $g\in C^{1,2}_{buc} $ and hence $g(t)\in \mathcal C^{(1+\gamma)+}$ for any $\gamma\in(0,1)$. The hypothesis on $h$ are satisfied since $\nabla h\in \mathcal C^{(1-\beta)-}$ implies $\nabla h\in \mathcal C^{\beta+}$.

For the (uniform) time-continuity with values in $ \mathcal C^{(1+\beta)+}$, since $\beta$ is not an integer,  we have to control \begin{equation}\label{eq:3terms}
\|f(t)-f(s)\|_\infty + \|\nabla f(t)-\nabla f(s)\|_\alpha
\end{equation}
for some $\alpha>\beta$ and  for small $|t-s|$, where we recall $f(t) = \tilde f(t, \phi(t, \cdot))$, having used the equivalent norm
\cite[(2.3)]{issoglio_russoPDEa}. 
The first term in \eqref{eq:3terms} is obvious from the fact that $\tilde f\in C^{1,2}_{buc}$ and  
\begin{equation}\label{eq:phiu}
\phi(t, x) -\phi(s,x) = u(t, x) -u(s,x), \, \text{ where } u \in C_T \mathcal C^{1+\alpha},
\end{equation}  
see Section \ref{ssc:pointwise}.

For the second term in \eqref{eq:3terms}, setting $H:= \nabla \tilde f \circ \phi$, we can write 
$\nabla f = H \nabla \phi$.
We note that $\nabla \phi \in C_T\mathcal C^\alpha$, see  Section \ref{ssc:pointwise}, and since $H \in C_T\mathcal C^\alpha$ (proved below) then the product is also in $C_T\mathcal C^\alpha$ and the proof is concluded.

It remains to show that $H \in C_T\mathcal C^\alpha$. For the sup part of the norm (see \cite[(2.2)]{issoglio_russoPDEa}) we notice that
\begin{align}\label{eq:HH}
 \nonumber
 \|H(t) - H(s)\|_\infty & \leq   \|\nabla \tilde f(t, \phi(t, \cdot)) - \nabla \tilde f(t, \phi(s, \cdot)) \|_\infty +  \|\nabla \tilde f(t, \phi(s, \cdot)) - \nabla \tilde f(s, \phi(s, \cdot)) \|_\infty\\
 &\leq  \|\text{Hess}(\tilde f) \|_{\infty}\|\phi(t, \cdot) -\phi(s, \cdot)\|_\infty +  \|\nabla \tilde f(t, \phi(s, \cdot)) - \nabla \tilde f(s, \phi(s, \cdot)) \|_\infty,
\end{align}
and the first term is bounded  as above using \eqref{eq:phiu}, while the second term is  controlled
because $\tilde f \in C^{1,2}_{buc}$.

We observe that $H\in C^{0,1}$ and $\nabla H= (\text{Hess}(\tilde f)\circ \phi ) \nabla \phi$. We will use below that $\nabla H$ is uniformly continuous, which we see by showing that each term of the product is bounded and uniformly continuous ({\it buc}). $\nabla \phi$ is {\it buc} because $\nabla\phi \in C_T \mathcal C^{\alpha}$. The term $\text{Hess}(\tilde f)\circ \phi$ is similar to \eqref{eq:HH} but using that Hess$(\tilde f)$ is {\it buc} and  \eqref{eq:phiu}.

Concerning the $\alpha$-seminorm (see \cite[(2.2)]{issoglio_russoPDEa}),
for $x_1, x_2$ such that $|x_1 -x_2|<1$ we have 
\begin{align*}
& \frac{|H(t, x_1) - H(t, x_2) - \left ( H(s, x_1) - H(s, x_2) \right) |}{|x_1-x_2|^\alpha} \\
\leq & 
\int_0^1 \left|\nabla H(t, x_1 + a (x_2-x_1)) -  \nabla H(s, x_1 + a (x_2-x_1)) \right| \di a   |x_2 - x_1|^{1-\alpha}   \\
 \leq &  \omega_{\nabla H}( |t-s|) ,
\end{align*}
where $\omega_{\nabla H} (\cdot)$ denotes the continuity modulus of $ \nabla H$. This concludes the control of the second term in \eqref{eq:3terms}.
\end{proof}

\begin{lemma}\label{lm:LL}
Let $ \tilde f\in C^{1,2}_{buc}$ and $\phi$ be the unique solution of PDE \eqref{eq:PDEphi}. Setting $ f  :=  \tilde f \circ \phi $ we have $f\in \mathcal D_{\mathcal L}^0$ and 
\begin{equation*}
(\tilde {\mathcal L} \tilde f ) \circ \phi=  {\mathcal L}  f 
\end{equation*}
in $C_T \mathcal C^{0+}$, that is $f$ is a solution of $\mathcal L f =g$ with $g:=(\tilde {\mathcal L} \tilde f ) \circ \phi \in C_T \mathcal C^{0+}$. Equivalently, we have $\tilde {\mathcal L} \tilde f = (  {\mathcal L}  f) \circ \psi $, where $\psi$ is the space-inverse of $\phi$ defined in \eqref{eq:psi}.

If moreover $\tilde f$ has compact support, then $f(T)$ and $g$ also have compact support, in which case $f\in \mathcal D_{\mathcal L} $.
\end{lemma}

\begin{proof}
We start by proving that  $f\in {\mathcal D}_{\mathcal L}^0$ so that we can then calculate $\mathcal L f$.
Notice that $f\in C_T \mathcal C^{(1+\beta)+}$ by  Lemma \ref{lm:tildef}. To show that  $f\in C^1([0,T], \mathcal S')$ we compute the time-derivative $\dot f $. 
Recall that  $\tilde f \in C^{1,2}_{buc}$ by assumption, and that $\phi:[0,T]\times \R^d \to \R^d$ and $f, \tilde f: [0,T]\times \R^d \to \R$. We have
\begin{equation}\label{eq:fdot}
t \mapsto \dot f (t, \cdot) = \dot{  \tilde f }(t, \phi(t, \cdot)) + \sum_{k=1}^d \partial_k \tilde f(t, \phi(t, \cdot)) \dot \phi_k (t, \cdot)  ,
\end{equation}
where the dot $\dot{\ }$ denotes  the time-derivative  and $\partial_k:= \frac{\partial }{ \partial x_k}$. 
We show that  the right-hand side of equation \eqref{eq:fdot} is  in $C_T\mathcal S'$. 
For the first term in  \eqref{eq:fdot} clearly we have the claim because   $ \dot{\tilde{f}} \circ \phi$ is uniformly continuous  in $t,x$. 
The second term in  \eqref{eq:fdot}
has products of the form  $(\partial_k \tilde f \circ \phi ) \dot \phi_k$ where $\dot \phi_k \in C_T \mathcal C^{(-\beta)-}$, see Section \ref{ssc:pointwise} and $\partial_k \tilde f \circ \phi \in C_T \mathcal C^{\beta+}$. Hence the product is well-defined and continuous by  \eqref{eq:bonyt}.
This  shows that $f\in C^1([0,T]; \mathcal S')$ and hence $f\in {\mathcal D}_{\mathcal L}^0$. 

 We now apply $\mathcal L$ to $f$ so we need to calculate the spatial derivatives of $f$. 
The first space derivative of $f$ with respect to $x_i$ is
\begin{align*}
\partial_i f (t, \cdot) = \sum_{k=1}^d  \partial_k  \tilde f (t, \phi(t, \cdot)) \partial_i \phi_k (t, \cdot) , t\in[0,T]
\end{align*}
and the second derivative is 
\begin{align*}
\partial_{ii} f (t, \cdot) &=  \sum_{k=1}^d\left [ \sum_{l=1}^d ( \partial_{lk}  \tilde f (t, \phi(t, \cdot)) \partial_i \phi_l (t, \cdot) ) \partial_i \phi_k (t, \cdot)  + \partial_k \tilde f (t, \phi(t, \cdot))  \partial_{ii} \phi_k(t, \cdot) \right]\\
&=   \left ((\nabla \phi)^T (\text{Hess}(\tilde f)\circ \phi)  \nabla \phi  \right)_{ii}(t, \cdot)+ \sum_{k=1}^d  \partial_k \tilde f (t, \phi(t, \cdot))  \partial_{ii} \phi_k(t, \cdot) , t\in[0,T].
\end{align*}
Note that $\partial_i f (t, \cdot)$ for all $ t\in[0,T]$  is a well-defined object in $\mathcal C^{(-\beta)-}$ because it is actually a function in $ \mathcal C^{\beta+}$ by  Lemma \ref{lm:tildef}.
The second derivative   $\partial_{ii} f (t, \cdot) $ is made of two terms: the first one  is  a bounded function,  and the second one is well-defined in $\mathcal C^{(-\beta)-}$ again by means of the pointwise product \eqref{eq:bony}, where for all  $ t\in[0,T]$  the distributional term $\partial_{ii}\phi_k (t, \cdot)$ is in $\mathcal C^{ (-\beta)-}$  since $\partial_{i}\phi_k (t, \cdot) \in \mathcal C^{ (1-\beta)-} $, see Section \ref{ssc:pointwise}. Using these we calculate $\mathcal Lf$  :
\begin{align}\label{eq:Lf} 
\nonumber
(\mathcal L f) (t, \cdot)
=& \dot  {\tilde f} (t, \phi(t, \cdot)) + \frac12 \sum_{i=1}^d \left ( (\nabla \phi)^T (\text{Hess}(\tilde f)\circ \phi) \nabla \phi \right)_{ii}(t,\cdot)\\
&+ \sum_{k=1}^d  \partial_k \tilde f (t, \phi(t, \cdot)) \left[ \dot \phi_k (t, \cdot) + \frac12  \sum_{i=1}^d\partial_{ii} \phi_k(t, \cdot) +\partial_i \phi_k (t, \cdot)  b_i(t, \cdot) \right], t\in [0,T],
\end{align}
where the last term $\partial_k \tilde f (t, \phi(t, \cdot)) \partial_i \phi_k (t, \cdot)  b_i(t, \cdot)$ is well-defined in $\mathcal C^{(-\beta)+}$  by \eqref{eq:bony}  used twice. Thus  equality \eqref{eq:Lf} holds  in the space $\mathcal C^{(-\beta)-}$.
Now we observe that $\mathcal L \phi_k   = \lambda (\phi_k    -\text{id}_k)$ 
because $\phi_k$ is  solution of PDE \eqref{eq:PDEphi}, see Section \ref{ssc:pointwise}.
 Thus the equality above becomes
\begin{align}\label{eq:Lf2}
\nonumber
(\mathcal L f) (t, \cdot)
=& \dot  {\tilde f} (t, \phi(t, \cdot)) + \frac12 \sum_{i=1}^d \left ((\nabla \phi)^T (\text{Hess}(\tilde f)\circ \phi)  \nabla \phi \right)_{ii}(t, \cdot)\\ \nonumber
&+ \sum_{k=1}^d  \partial_k \tilde f (t, \phi(t, \cdot)) \lambda (\phi_k (t, \cdot)-\text{id}_k)\\ \nonumber
=& \dot  {\tilde f} (t, \phi(t, \cdot)) + \frac12 \text{Tr} \left (  (\nabla \phi)^T (\text{Hess}(\tilde f)\circ \phi) \nabla \phi\right)(t, \cdot)\\
& + \lambda  \nabla \tilde f (t, \phi(t, \cdot))  (\phi(t, \cdot)-\text{id}) , t \in [0,T].
\end{align}
 On the other hand,  by direct definition \eqref{eq:Ltilde} of $\tilde {\mathcal L}$ applied to $\tilde f \in C^{1,2}_{buc}$ and then composed with $\phi$ and using $\psi(t, \phi(t, \cdot) = \text{id}$,  one easily gets 
\begin{align}\label{eq:tLf}
\nonumber
(\tilde {\mathcal L} \tilde f )(t, \phi(t, \cdot))
=& \dot{  \tilde f} (t, \phi(t, \cdot)) + \frac12 \text{Tr} \left ( (\nabla \phi)^T (\text{Hess}(\tilde f)\circ \phi)  \nabla \phi \right)(t, \cdot)\\
& + \lambda \nabla \tilde f (t, \phi(t, \cdot))  (\phi(t, \cdot)-\text{id}), t\in[0,T].
\end{align}
Now using \eqref{eq:Lf2} and \eqref{eq:tLf} we get $t\mapsto (\mathcal L f)=(\tilde {\mathcal L} \tilde f )(t, \phi(t, \cdot)) $ in $C([0,T];\mathcal S')$. 
We observe that the right-hand side of \eqref{eq:tLf} belongs to $C_T \mathcal C^{0+}$. 
Setting $g:=(\tilde {\mathcal L} \tilde f )\circ \phi$ we can conclude that  $\mathcal L f =g \in C_T \mathcal C^{0+}$. Given that both sides are functions, we can compose them with $\psi $ to get $\tilde {\mathcal L} \tilde f=  ({\mathcal L}  f  ) \circ \psi $.

Finally we show  that  if $\tilde f$ has compact support, then $g=(\tilde{\mathcal  L} \tilde f ) \circ \phi $ also has compact support. First notice that $ \tilde {\mathcal  L} \tilde f $ has compact support, thus there exists $M>0$ such that for all $(t, x)$ with $|(t, \phi(t, x))|>M$ then $g(t,x)=0$. To show that $g$ has compact support it is enough to find $N>0$ such that if $|(t,x)|> N$, 
then $|(t,\phi(t,x))|>M$. This is equivalent to showing that 
\[
A:= \{(t,x): |(t,\phi(t,x))|\leq M\} \subset \{(t,x): |(t,x)|\leq N \}=:B,
\] for some $N$. To show the above inclusion, let $(t,x)\in A$. 
We write $(t, x) = (t, \psi(t, \phi(t,x)))$ and using that $\nabla \psi$ is uniformly bounded, see Section \ref{ssc:pointwise}, we get
\begin{align*}
|(t,x)|& = |(t,\psi(t, \phi(t,x)) - \psi(t, 0) +\psi(t, 0))|\\
&\leq C |(t,\phi(t,x))| +|(t, \psi(t, 0))|\\
&\leq C M +\sup_{t\in [0,T]}| (t,\psi(t, 0))| =: N,
\end{align*}
which shows that $(t,x)\in B$. 
We conclude by  noting that  $f(T, \cdot)$ also has compact support, following  the above computations but fixing the time $t=T$ and replacing $\tilde {\mathcal L} \tilde f $ with $\tilde f$.
\end{proof}

\begin{lemma} \label{lm:DL}
 We have $C^{1,2}_{c} \subset {\tilde {\mathcal D}}_{\tilde{\mathcal L}}$. 
\end{lemma}

\begin{proof}
By Definition of  ${\tilde {\mathcal D}}_{\tilde{\mathcal L}}$ we have  to show that if $\tilde f\in C^{1,2}_{c}$ then $ f := \tilde f \circ \phi \in \mathcal D_{\mathcal L}$, where $ \mathcal D_{\mathcal L}$ is given in \eqref{eq:D}. First we note that by Lemma \ref{lm:tildef} we have $f \in C_T\mathcal C^{(1+\beta)+}$. 
Next we show that $\mathcal Lf =g $ for some $g\in C_T\bar{\mathcal C}_c^{0+}$. We define $g:= \tilde{\mathcal L} \tilde f \circ \phi$. By Lemma \ref{lm:LL} we have $\mathcal L f =g $ and since $\tilde f$ has   compact support, then  $f\in \mathcal D_{\mathcal L}$ by Lemma \ref{lm:LL} again. 
\end{proof}

We can finally state the main result of this section, namely the equivalence between the original martingale problem  and the Zvonkin-transformed martingale problem.

\begin{theorem}\label{thm:mart}
Let Assumption \ref{ass:param-b} hold. 
\begin{itemize}
\item[(i)] If $(X, \mathbb P)$ is a solution to  MP with distributional drift $b$ and i.c.\ $\mu$ then $(Y, \mathbb P) $ is a solution in law  to \eqref{eq:SDEY},   where $Y_t := \phi(t, X_t)$ and  $Y_0\sim \nu$, where  $\nu$ is the pushforward measure of $\mu$ given by
$\nu:= \mu (\psi(0, \cdot))$. 
\item[(ii)] If $(Y, \mathbb P) $ is a solution in law  to \eqref{eq:SDEY} with $Y_0\sim \nu$ then $(X, \mathbb P)$  is a solution to MP with distributional drift $b$ and i.c. $\mu$, where $X_t := \psi(t, Y_t)$ and $\mu$ is the pushforward measure of $\nu$ given by $\mu:= \nu (\phi(0, \cdot))$.
\end{itemize}
\end{theorem}
\begin{proof}
\emph{Item (i).} 
Let $(X, \mathbb P)$ be a solution of MP. For any $\tilde f \in C_c^\infty $ we define $f:= \tilde f \circ \phi $, where $\phi$ is the unique solution of  PDE \eqref{eq:PDEphi}.
 By Lemma \ref{lm:LL} $f \in \mathcal D_{\mathcal L}$. 
Setting $Y_t:= \phi(t, X_t)$,  by Lemma \ref{lm:LL}  we have 
\[
\tilde f (Y_t) -\tilde f (Y_0) - \int_0^t (\tilde{\mathcal L} \tilde f) (s, Y_s) \mathrm ds =  f(t, X_t) - f(0,X_0) - \int_0^t ( \mathcal L f)(s, X_s) \mathrm ds , 
\]
which is a local martingale under $\mathbb P$ for all $\tilde f\in C^\infty_c$  by  Definition \ref{def:MP} since $f\in \mathcal D_{\mathcal L}$. It follows that the couple $(Y, \mathbb P)$ satisfies the Stroock-Varadhan martingale problem, therefore  $(Y, \mathbb P)$ is a solution in law of SDE \eqref{eq:SDEY}.

\emph{Item (ii).} 
Let $(Y, \mathbb P)$ be a solution in law of SDE \eqref{eq:SDEY}. We define $X_t := \psi(t, Y_t)$, where $\psi$ is the (space-)inverse of $\phi$   defined in \eqref{eq:psi}.
To show that $(X, \mathbb P)$ is a solution to MP with distributional drift $b$ we need to show that for all $f\in \mathcal D_{\mathcal L}$ the quantity
\[
f(t, X_t) - f(0, X_0) -\int_0^t (\mathcal L f)(s, X_s) \mathrm d s
\]
is a local martingale under $\mathbb P$. Since $ f\in \mathcal D_{\mathcal L} $ then there exists $g \in C_T \mathcal C^{0+}$ (so  there exists $\nu \in(0,1)$ with $g\in C_T \mathcal C^\nu$) such that $\mathcal L f = g$.
 We define  $\tilde g:= g\circ \psi$, $\tilde f_T:= f (T, \psi(T, \cdot))$ and  $\tilde f^n_T : = \tilde f_T \ast \rho_n$, where $\rho_n=p_{\frac1n}$ with $p_t$ the heat kernel.
We see that $ \tilde g\in \mathcal C^{0,\nu}$, see Appendix \ref{app:lunardi} for the explicit definition of the space.
 Indeed  $\tilde g$ is in $C([0,T]\times \R^d)$ because $g$ and $\psi$ are, and it is easy to obtain the bound 
\[
 \sup_{t\in[0,T]}  \sup_{x\neq y} \frac{|\tilde g(t, x)- \tilde g(t, y)| }{|x-y|^\nu} \leq \sup_{t\in[0,T]}  \|g(t)\|_{\mathcal C^\nu} \|\nabla \psi \|_\infty^\nu, 
\]
using the fact that $g\in C_T\mathcal C^\nu$ and $\psi \in C^{0,1}$ with gradient $\nabla \psi$ uniformly bounded, see Section \ref{ssc:pointwise}. Moreover $\tilde f^n_T\in C^{2+\nu}$ (for explicit definition of these spaces  and its inclusion in other spaces,  see Appendix \ref{app:lunardi}) and  by Remark \ref{rm:tildeL}   the coefficients of $\tilde{\mathcal L}$  are in $C^{0,\nu}$.
 So  by \cite[Theorem 5.1.9]{lunardi95} (which has been recalled in  Theorem \ref{thm:lunardi} in the Appendix for ease of reading) we know that for each $ n$  there exists  a function $\tilde f^n \in C^{1, 2+\nu}([0,T]\times \R^d)$ (see Appendix \ref{app:lunardi} for the definition of this space and its inclusion in other spaces) which is the classical solution of 
\begin{equation}\label{eq:PDEtildefn}
\left\{
\begin{array}{l}
\tilde{\mathcal L} \tilde f^n = \tilde g\\
\tilde f^n(T) = \tilde f^n_T.
\end{array}
\right.
\end{equation}
Therefore $\tilde f^n \in C^{1,2}$  and thus by It\^o's formula
\[
\tilde f^n(t, Y_t) - \tilde f^n(0, Y_0) -\int_0^t \tilde g(s, Y_s) \mathrm d s
\]
is a local martingale under $\mathbb P$. 
Here we used that   $ (\tilde { \mathcal L} \tilde f_n)(s, Y_s) = \tilde g(s, Y_s)$ by construction.  Setting $f^n:= \tilde f^n \circ \phi$ we also have  that 
\begin{equation}\label{eq:fn}
 f^n(t, X_t) -f^n(0, X_0) -\int_0^t g(s, X_s) \mathrm d s
\end{equation}
is a local martingale under $\mathbb P$. 
 Using the definition of $\tilde g$, the fact that $\tilde f^n$ is a classical solution of PDE \eqref{eq:PDEtildefn} and  $\tilde f^n \in C^{1,2}_{buc}$ (see Remark \ref{rm:C12buc})  by Lemma \ref{lm:LL}  we know that 
 \[
g= \tilde g \circ \phi = {\tilde {\mathcal L}} \tilde f^n \circ \phi = \mathcal L f^n,
 \] 
in $C_T\mathcal C^{\nu}$ and thus in particular $f^n$ is a weak solution of 
\begin{equation}\label{eq:PDEfn}
\left\{
\begin{array}{l}
{\mathcal L} f^n =  g\\
 f^n(T) =  f^n_T,
\end{array}
\right.
\end{equation}
 where  $f^n_T := \tilde f^n(T) \circ \phi (T, \cdot)$. 

Now we claim that $f^n$ is the unique mild solution to \eqref{eq:PDEfn} in $C_T\mathcal C^{(1+\beta)+}$ and that $f^n\to f$ uniformly on compacts (these claims will be proven later). By this convergence and taking the limit of \eqref{eq:fn} where we  replace $g= \mathcal L f$, we get that
\[
f(t, X_t) - f(0, X_0) -\int_0^t (\mathcal L f)(s, X_s) \mathrm d s
\]
is a local martingale under $\mathbb P$, thanks to the fact that the space of local martingales is closed under u.c.p.\ convergence.

It is left to prove that  $f^n$ is the unique mild solution to \eqref{eq:PDEfn} in $C_T\mathcal C^{(1+\beta)+}$ and that $f^n\to f$ uniformly on compacts, which we do   in three steps.\\
{\em Step 1: we prove that $f^n$ is the unique mild solution to \eqref{eq:PDEfn} in $ C_T\mathcal C^{(1+\beta)+} $.} To do so, first we show that $f^n\in C_T\mathcal C^{(1+\beta)+} $, indeed   $f^n: =  \tilde f^n \circ \phi$ with $\tilde f^n\in C_{buc}^{1,2}$ and $\phi$ solution of PDE \eqref{eq:PDEphi}, so by  Lemma \ref{lm:tildef} we have $f^n\in C_T\mathcal C^{(1+\beta)+} $. In Section \ref{ssc:pointwise} it is recalled that weak and mild solutions are equivalent  therefore  $f^n$ is the unique (mild) solution  in  $C_T\mathcal C^{(1+\beta)+}$. \\
{\em Step 2: we prove that   $f_T^n \to f_T: = f(T)$ in $\mathcal C^{(1+\beta)+}$.} Recall that $f^n_T = \tilde f^n_T \circ \phi (T, \cdot)$, so by Lemma \ref{lm:tildef} again we have $f^n_T\in\mathcal C^{(1+\beta)+} $. Moreover  $f_T = f(T)\in \mathcal C^{(1+\beta)+} $ because $f\in \mathcal D_{\mathcal L}$. 
Now we notice that $\tilde f_T \in \mathcal C^{(1+\beta)+} $  by Lemma \ref{lm:gh}  using the definition  $\tilde f_T := f_T \circ \psi(T, \cdot) $, where $f_T\in \mathcal C^{(1+\beta)+}$ by definition of $\mathcal D_{\mathcal L} $ and $\psi( T, \cdot) \in C^1$ with $\nabla \psi (T, \cdot) \in \mathcal C^{(1-\beta)-}$ see Section \ref{ssc:pointwise}.
Since  $\tilde f_T^n =  \tilde  f_T \ast \rho_n$ and the convolution with the mollifier $\rho_n$ maintains the same regularity of $\tilde f_T$ by \cite[Lemma 2.4]{issoglio_russoPDEa},
then  $\tilde f^n_T \to \tilde f_T$ in $\mathcal C^{(1+\beta)+}$, see Section \ref{ssc:pointwise}. Finally again by  Lemma \ref{lm:gh} we have  $\tilde f^n_T \circ \phi(T, \cdot) \to \tilde f_T \circ \phi(T, \cdot)$ in $\mathcal C^{(1+\beta)+}$ as wanted.\\
{\em Step 3: we prove that  $f^n\to f$ uniformly, in particular uniformly  on compacts.} From Step 1 we have  that $f^n$ is the unique solution of \eqref{eq:PDEfn} in $C_T C^{(1+\beta)+}$. Moreover we recall that  $f$ is the unique mild solution in the same space of $\mathcal L f = g$ with terminal condition the value of the function itself, $ f_T = f(T)$. 
We can now apply continuity results on the PDE \eqref{eq:PDEfn}, see Section \ref{ssc:pointwise}, to conclude that  $f^n\to f$ in $C_T \mathcal C^{(1+\beta)+}$.
 This clearly implies that $f^n\to f$ uniformly, as wanted.
\end{proof}

\begin{remark}
It is possible to define an equivalent MP by a  transformation different than the one used in Theorem \ref{thm:mart}. Indeed, it is enough to consider a generic transformation $\phi\in C_T D \mathcal C^{\beta+}$ which is space-invertible with inverse $\psi$, and under which one has the equivalence between  $(X, \mathbb P)$ solving  the MP with respect to $\mathcal L$  and $(\phi (X), \mathbb P)$ solving the MP with respect to $\tilde{\mathcal L}$, where $\tilde{ \mathcal  L} \tilde f := \mathcal L f \circ \psi$. The issue going further would be to interpret $\tilde {\mathcal L} \tilde f = \tilde g$ as a PDE, which   would  need to be considered in the mild sense and will presumably require some regularity of $\phi$. Well-posedness of such an equation would be based on Schauder-type estimates for the time-dependent semigroup generated by the diffusive component of the operator $ \tilde{\mathcal L}$, which are far from being straightforward.
\end{remark}

From now on, let $(b^n)$ be the sequence defined in \cite[Proposition 2.4]{issoglio_russoMK}, so we know that $b^n\to b$ in $C_T\mathcal C^{-\beta}$, $b^n \in C_T \mathcal C^\gamma$ for all $\gamma \in \R$ and $b^n$ is bounded and Lipschitz.   Recall that $\lambda>0$ has been  fixed and independent of $n$, chosen such that \eqref{eq:lambda} holds.
To conclude the section, we prove  a continuity result for the transformed problem for $Y$ that will be useful when we will prove analogous continuity results for the original problem for $X$. Let us denote by $Y^n$ the strong solution of 
\begin{equation}\label{eq:Yn}
Y^n_t= \phi(0, X_0) + \lambda \int_0^t Y^n_s \di s - \lambda\int_0^t \psi^n(s, Y^n_s) \di s + \int_0^t \nabla\phi^n(s, \psi^n(s, Y^n_s) )\di W_s,
\end{equation}
which is the counterpart of \eqref{eq:SDEY} when one replaces $b$ with $b^n$.
 
\begin{remark}\label{rm:Yn}
We notice that the drift and the diffusion coefficient of \eqref{eq:Yn} are uniformly bounded in $n$. Indeed the drift is given by  $\lambda (y -\psi^n(s, y)) = \lambda  u^n (s, \psi^n(s,y))$ and the diffusion coefficient is $ \nabla\phi^n(s, \psi^n(s, y)) =   \nabla u^n(s, \psi^n(s, y)) + I_d$.
Thanks \cite[Lemma 4.9]{issoglio_russoPDEa}, for every fixed $\alpha \in (\beta, 1-\beta)$ we have
\[
\| u^n \|_{ C_T \mathcal C^{\alpha+1}}  \leq R_\lambda( \| b^n\|_{ C_T \mathcal C^{-\beta}}  )   \|b^n\|_{C_T \mathcal C^{-\beta}} \leq R_\lambda( \sup_n\| b^n\|_{ C_T \mathcal C^{-\beta}}  ) \sup_n \|b^n\|_{C_T \mathcal C^{-\beta}},
\] 
where $R_\lambda$ is an increasing function.
Thus $u_n$  
 and $\nabla u_n$ are uniformly bounded in $n$.
\end{remark}

\begin{lemma}\label{lm:Yntight}
  Let $Y^n$ be the solution of SDE \eqref{eq:Yn}. Then the sequence
  of laws of   $(Y^n)$ is tight.
\end{lemma}
\begin{proof}
According to \cite[Theorem 4.10 in Chapter 2]{karatzasShreve} we need to prove that 
\begin{equation}\label{eq:KS1}
\lim_{\eta\to\infty} \sup_{n\geq 1} \mathbb P(| Y^n_0 |>\eta ) =0
\end{equation}
and that for every $\varepsilon>0$
\begin{equation}\label{eq:KS2}
\lim_{\delta\to0} \sup_{n\geq 1} \mathbb P \Big ( \sup_{\substack{
s,t \in [0,T] \\|s-t|\leq \delta }} |Y^n_t-Y^n_s|>\varepsilon   \Big ) =0 .
\end{equation}
We know that $ Y^n_0 =  \phi^n(0, X_0)$ and $X_0\sim \mu$. By continuity results on the PDE \eqref{eq:PDEphi}, see Section \ref{ssc:pointwise}, we have that $\phi^n\to\phi $ uniformly and   that 
\[a:= \sup_{n\geq 1} \|\nabla \phi^n\|_\infty<\infty  \quad \text{and} \quad  b:= \sup_{n\geq 1} |\phi^n(0,0)|<\infty .\]
So the first condition \eqref{eq:KS1} for tightness gives
\begin{align*}
\mathbb P(| Y^n(0) |>\eta ) &= \mathbb P(| \phi^n(0, X_0) |>\eta ) \\
&\leq \mathbb P(| \phi^n(0,0) | + \|\nabla \phi^n\|_\infty |X_0| >\eta ) \\
&\leq \mathbb P(a + b |X_0| >\eta ).
\end{align*}
Noticing that $a+b|X_0| $ is a finite random variable (independent of $n$) then we have \eqref{eq:KS1}.

Concerning the second bound \eqref{eq:KS2} for tightness, we first observe that the classical Kolmogorov criterion 
\begin{equation}\label{eq:kolm}
\E[|Y_t^n - Y_s^n|^4] \leq C |t-s|^2
\end{equation}
holds for some positive constant $C$ independent of $n$. The proof of this bound works exactly as the the proof in \cite[Step 3 of Proposition 29]{flandoli_et.al14}: indeed, the process $Y^n$ therein has the same form as $Y^n$ given by \eqref{eq:Yn}.
By Remark \ref{rm:Yn} we have that the drift and diffusion coefficients are uniformly bounded  in $n$, so that  \cite[Step 3 of Proposition 29]{flandoli_et.al14}  allows to  show \eqref{eq:kolm}.
 
 Now we apply Garsia-Rodemich-Rumsey Lemma (see e.g.\ \cite[Section 3]{BarlowYor}) and we know that for every $0<m<1$ there exists a constant $C'$ and a random variable $\Gamma_n$ such that 
\[
| Y_t^n - Y_s^n|^4 \leq C' |t - s|^m \Gamma_n
\]
with 
\begin{equation}\label{eq:Gamma}
\mathbb E(\Gamma_n) \leq c\  C \frac1{1-m} T^{2-m},
\end{equation}
where $c$ is a universal constant. 
Consequently, for every $\varepsilon>0$ and for every $n\geq1$
\begin{align*}
\mathbb P \Big ( \sup_{\substack{s,t \in [0,T] \\|s-t|\leq \delta }} |Y^n_t-Y^n_s|>\varepsilon   \Big ) 
=&  \mathbb P \Big ( \varepsilon < \sup_{\substack{
s,t \in [0,T] \\|s-t|\leq \delta }}  |Y^n_t-Y^n_s| \leq   C'^{\frac14} \delta^{\frac m4} \Gamma_n^{\frac14} \Big )\\
\leq &   \mathbb P \Big (  \varepsilon  \leq   C'^{\frac14} \delta^{\frac m4} \Gamma_n^{\frac14} \Big )\\
\leq &   \mathbb P \Big ( \Gamma_n   \geq  \frac {\varepsilon^4}{ C'\delta^{m} } \Big )\\
\leq &  \frac { C'\delta^{m} } {\varepsilon^4} \E(\Gamma_n),
\end{align*}
by Chebyshev inequality.
So, using \eqref{eq:Gamma} we have that $\sup_{n\geq 1} \mathbb P \Big ( \sup_{\substack{
s,t \in [0,T] \\|s-t|\leq \delta }} |Y^n_t-Y^n_s|>\varepsilon   \Big ) \to 0$ as $\delta\to0$ and \eqref{eq:KS2} is established.
\end{proof}

\begin{remark}\label{rm:SW}
When $Y_0=y $ is a deterministic initial condition, we know that \eqref{eq:SDEY} admits  existence  and uniqueness in law by \cite[Theorem 10.2.2]{stroock_varadhan},   because the drift  and diffusion coefficient are 
bounded by Remark \ref{rm:Yn} and the diffusion coefficients is continuous since $\nabla \phi$ and $\psi$ are continuous and it is uniformly non degenerate since $\|\nabla u\|_\infty \leq \frac12$, see Section \ref{ssc:pointwise}.
\end{remark}

\section{The martingale problem for $X$}\label{sc:MP}

In this section we solve the martingale problem for the process $X$, which formally satisfies an SDE of the form
\[
X_t = X_0 + \int_0^t b(s, X_s) \di s +  W_t,
\]
where $W$ is a $d$-dimensional Brownian motion, the drift $b$ is an element of $C_T\mathcal C^{(-\beta)+}$  that satisfies  Assumption \ref{ass:param-b}  and the initial condition $X_0$ is a given random variable. To do so, we first solve the problem for a deterministic initial condition and then we use this to extend the result to any initial condition. 
We also derive some properties about said solution, such as its link to the Fokker Planck equation and continuity properties.

\vspace{10pt}

We start with the case when the drift $b$ is a function, by comparing the notion of solution to the singular MP   with the notion of  solution in law of SDEs, and with the Stroock-Varadhan Martingale Problem, see \cite[Section 6.0]{stroock_varadhan}.  We recall that $(X, \mathbb P)$ is a solution to  the Stroock-Varadhan Martingale Problem  with respect to $\mathcal L$ if for every $f\in C_c^\infty$ 
\begin{equation}\label{eq:SV}
f( X_t) - f(X_0) - \int_0^t (\frac12 \Delta f (X_s) + \nabla f(X_s)  b(s, X_s)) \di s 
\end{equation}
is a local martingale.

\begin{lemma}\label{lm:strongMP}
Let $b\in C_T \mathcal C^{0+}$. Let $(\Omega, \mathcal F, \mathbb P)$ be some probability space. Let  $X_0 \sim \mu$. Then the following are equivalent.
\begin{itemize}
\item[(i)] The couple $(X, \mathbb P)$ is solution to the MP with distributional drift $b$.
\item[(ii)] The couple $(X, \mathbb P)$ is solution to the  Stroock-Varadhan Martingale Problem with respect to $\mathcal L$.
\item[(iii)] There exists a Brownian motion $W$ such that the process  $X$ under $\mathbb P$ is a solution of $\di X_t = b(t, X_t) \di t + \di W_t$.
\end{itemize}
\end{lemma}
\begin{proof}
(ii) $\iff$ (iii). This follows from  the  Stroock-Varadhan  classical theory, see \cite[Chapter 8]{stroock_varadhan}. We sketch the proof for completeness.   If the Stroock-Varadhan Martingale Problem problem is fulfilled, i.e.\ if  (ii) holds, 
  then in fact \eqref{eq:SV} also holds for $f \in C^2$. Choosing $f(x) = x^i$ and
  $f(x) = x^i x^j , 1 \le i,j \le d,$ one can show that
  $M_t = X_t - X_0 - \int_0^t b(s,X_s) \di s$ is a local martingale
  with covariation matrix $([X^i, X^j])_{i,j}$ being the identity.
  The process $M$ is then a standard $d$-dimensional Brownian
  motion by L\'evy's characterization theorem.
  Viceversa if $X$ fulfills the SDE (iii) then
  (ii) follows by It\^o's formula.
  \\
  (i) $\Longrightarrow$ (ii).   For this it is enough to show that for every $f\in C_c^\infty$  \eqref{eq:SV} holds. This is true since $C_c^\infty \subset \mathcal D_{\mathcal L}$ in this case.\\
  (iii) $\Longrightarrow$ (i). 
We will make use of the spaces  $C^{0,\nu}([0,T] \times \mathbb R^d)$ and $  C^{1, 2+\nu}$ for $\nu \in (0,1)$, which have been defined in  Appendix \ref{app:lunardi}. 
Since $ b \in C_T {\mathcal C}^{0+}$, by \cite[Remark 4.12]{issoglio_russoPDEa} we know that the unique solution $u\in C_T \mathcal C^{(1+\beta)+}$ of PDE \eqref{eq:PDEu} is also the classical solution as given in Theorem \ref{thm:lunardi}, hence  $u\in  C^{1,2}  $. 
We set $\phi = \text{id} + u$, which thus belongs to $ C^{1, 2}$ so by It\^o's formula applied  to $ Y= \phi (t, X_t)$ where $X$ is a solution to $\di X_t = b(t, X_t) \di t + \di W_t$  we get that $Y$ solves \eqref{eq:SDEY} with initial condition $Y_0 \sim \nu:= \mu(\psi(0, \cdot))$, where $\psi $ is the inverse of $\phi$.  
 Thus Theorem \ref{thm:mart} implies that $(X, \mathbb P) $ is a solution to the MP with (distributional) drift $b$ and i.c.\ $\mu$, as wanted. 
\end{proof}

Next we show the link between the law of the solution to the MP and the Fokker-Planck equation, in particular we show that   the law of the solution to the martingale problem with distributional drift  satisfies a Fokker-Planck equation.

\begin{theorem}\label{thm:McKvFP}
  Let Assumption \ref{ass:param-b} hold.
Let $(X,\mathbb P)$ be a solution to the martingale problem with distributional drift $b$ and initial condition $\mu$ with density $v_0$. Let $v(t, \cdot)$ be the law density of $X_t$ and let us assume that $v\in C_T\mathcal C^{\beta+}$. Then $v$ is  a weak solution of the Fokker-Planck equation, that is for every $\varphi\in \mathcal S$ we have 
\begin{equation}\label{eq:testPDE}
\langle \varphi, v(t)\rangle  =   \langle   \varphi, v_0\rangle + \int_0^t \langle \frac12 \Delta \varphi, v(s) \rangle \di s  + \int_0^t \langle  \nabla \varphi,    v (s) b(s) \rangle  \di s,  
\end{equation}
for all $t\in[0,T]$.
\end{theorem}
Notice that the product $v (s) b(s) $ appearing in the last integral is well-defined using pointwise products \eqref{eq:bony}. We remark that the solution $v$ is the unique solution of \eqref{eq:testPDE} by \cite[Theorem 3.7 and Proposition 3.2]{issoglio_russoMK}.

\begin{proof}
  It is enough to show the claim for all $\varphi \in C_c^\infty$.
  Indeed $C_c^\infty$ is  dense in $\mathcal S $. Since $\varphi\in C_c^\infty \subset \mathcal D^0_{\mathcal L}$, then  we  can apply the operator $\mathcal L$ defined in Definition \ref{def:L} to $\varphi$,  and we define $\mathcal L \varphi =:g $. Clearly $\varphi$ is a weak solution of the PDE $\mathcal L \varphi =g$ with terminal condition $\varphi$. Moreover the function $\varphi$ is time independent by construction.
Using the definition of $\mathcal L$ we get for all $s\in[0,T]$ that 
\begin{equation}\label{eq:Lphi}
(\mathcal L \varphi)(s) = \tfrac12 \Delta \varphi + \nabla \varphi \,  b(s) 
\end{equation}
in  $\mathcal C^{-\beta}$ (having used  the regularity of $\varphi$ and  the pointwise product \eqref{eq:bony}). In fact since  $t \mapsto b(t, \cdot)\in \mathcal C^{-\beta}$ is a continuous function of time by \eqref{eq:bonyt} we have that  $\mathcal L\varphi \in C_T\mathcal C^{-\beta}$.

We now construct  a sequence  $(g^n) \in C_T\mathcal C^{0+}$ that converges to $g$ in $C_T\mathcal C^{-\beta}$ and that is compactly supported. 
Let $(b^n)$ be the sequence
defined  before \eqref{eq:lambda},
in particular it converges to $b$ in   $C_T\mathcal C^{-\beta}$ and let us define
$ g^n:= \tfrac12 \Delta \varphi + \nabla \varphi\, b^n$. Then clearly $g^n\in C_T\mathcal C^{0+}$ (in fact it is more regular) and 
\[
\|g-g^n\|_{C_T \mathcal C^{-\beta}} = \| \nabla \varphi\, (b-b^n)\|_{C_T \mathcal C^{-\beta}} \leq \|\nabla \varphi\|_{C_T \mathcal C^{\beta+}} \|b-b^n\|_{C_T \mathcal C^{-\beta}},
\]
and the right-hand side goes to 0 as $n\to\infty$.
 Moreover, denoting by $K$ the compact support of $\varphi$, we have that also $g^n$ is supported on $K$.

Let us denote by $u^n$ the  mild  solution of $\mathcal L u^n  = g^n, \, u^n(T) = \varphi$, which exists and is unique in $C_T \mathcal C^{(1+\beta)+}$, see Section \ref{ssc:pointwise}. 
Such function belongs to $\mathcal D_{\mathcal L}$ by  definition of the domain $\mathcal D_{\mathcal L}$, see \eqref{eq:D}. 
Since $u^n \in \mathcal D_{\mathcal L}$ and $(X, \mathbb P)$ is a solution to the   martingale problem  with distributional
drift $b$ and initial condition $\mu$ with density $v_0$, then we know that $$u^n(t, X_t) - u^n(0, X_0) - \int_0^t \mathcal L u^n (s, X_s) \di s$$ is a local martingale under $\mathbb P$, but also a true martingale since $u^n$ and $\mathcal L u^n$ are bounded. We denoted by $v(t, \cdot)$  the law density of $X_t$, thus taking the expectation under $\mathbb P$ we have
\begin{align}\label{eq:Lu^n}
\int_{\R^d} u^n(t, x) v(t, x) \di x &- \int_{\R^d} u^n(0, x) v_0(x) \di x - \int_0^t \int_{\R^d}( \mathcal Lu^n) (s, x)  v(s, x) \di x \di s =0.
\end{align}

We now consider a smooth function $\chi_{K}\in C_c^\infty $ such that $\chi_{K} = 1$ on $K$.  Since $g^n$ is compactly supported on $K$ and $\mathcal Lu^n = g^n$, we can rewrite the double integral in \eqref{eq:Lu^n} as
\begin{align*}
\int_0^t \int_{\R^d}( \mathcal Lu^n) (s, x)  v(s, x) \di x \di s = & \int_0^t \int_{\R^d}( \mathcal Lu^n) (s, x)  v(s, x) \chi_{K}(x)\di x \di s\\
= & \int_0^t \langle ( \mathcal Lu^n) (s)  v(s),  \chi_{K}\rangle \di s,
\end{align*}
where the dual pairing is in $\mathcal S, \mathcal S'$.
By  continuity properties of  the PDE $\mathcal L u^n = g^n$ with terminal  condition $u^n(T)=\varphi$  (see Section \ref{ssc:pointwise})  we know that
since $g^n \to g $ in $C_T \mathcal C^{-\beta}$ then  $u^n \to \varphi $ in $C_T \mathcal C^{(1+\beta)+}$,  since  $\varphi$ is the unique  solution of $\mathcal L \varphi = g$ with terminal condition $u(T)= \varphi$. 
 Thus taking the limit as $n\to\infty$ of the above dual pairing we get
\begin{align}\label{eq:chi}
\lim_{n\to \infty} \int_0^t \langle ( \mathcal Lu^n) (s)  v(s),  \chi_{K}\rangle \di s  &= \int_0^t \langle ( \mathcal L\varphi) (s)  v(s),  \chi_{K}\rangle \di s\\ \nonumber
& = \int_0^t \langle \frac12\Delta \varphi \,  v(s),  \chi_{K}\rangle +\langle \nabla \varphi \, b(s)  v(s),  \chi_{K}\rangle  \di s  \\ \nonumber
& = \int_0^t \langle \frac12\Delta \varphi   , v(s) \rangle \di s +\int_0^t \langle \nabla \varphi \, b(s)  v(s),  \chi_{K}\rangle  \di s .
\end{align}
Now we prove that the latter dual pairing in \eqref{eq:chi} can be rewritten as 
\begin{equation}\label{eq:dualpair}
\langle \nabla \varphi \, b(s)  v(s),  \chi_{K}\rangle 
= \langle \nabla \varphi  , b(s)  v(s) \rangle,
\end{equation}
for all $s\in[0,T]$.
Indeed, the LHS of \eqref{eq:dualpair} is well-defined because $\chi_{K}\in C_c^\infty$ and  for every $s\in[0,T]$ the distribution $\nabla \varphi \, b(s)  v(s)$ is actually an element of $\mathcal C^{-\beta}$ because of the
pointwise product \eqref{eq:bony} and of the regularity $v(s)\in \mathcal C^{\beta+}$ and $b(s) \in \mathcal C^{-\beta}$. The RHS of \eqref{eq:dualpair} is also well-defined, but now the test function is $\nabla \varphi \in C_c^\infty$ and the distribution is $b(s)  v(s)$. 
To show that  \eqref{eq:dualpair} holds  we observe that by the continuity of the product \eqref{eq:bonyt} we have   $ b^n(s)  v(s) \to  b(s)  v(s)$ in $\mathcal C^{-\beta}$ (in fact uniformly in $s\in[0,T]$) and thus we can write
\begin{align*}
\langle \nabla \varphi \, b(s)  v(s),  \chi_{K}\rangle & = \lim_{n \to \infty} \langle \nabla \varphi\, b^n(s)  v(s),  \chi_{K}\rangle \\
&= \lim_{n \to \infty} \int_{\R^d} \nabla \varphi(x) b^n(s,x)  v(s,x) \chi_{K}(x) \di x\\
   &= \lim_{n \to \infty} \int_{\R^d} \nabla \varphi (x) b^n(s,x)  v(s,x)\di x\\
    &= \lim_{n \to \infty} \langle \nabla \varphi , b^n(s)  v(s) \rangle\\
&=\langle \nabla \varphi , b(s)  v(s) \rangle,
\end{align*}
for all $s\in[0,T]$, which proves \eqref{eq:dualpair}. 

To conclude it is enough to take the limit as $n\to \infty$ in \eqref{eq:Lu^n} and use \eqref{eq:chi} and \eqref{eq:dualpair} to get \eqref{eq:testPDE}.
\end{proof}

The following is a continuity result for the martingale problem. Recall that $(b^n)$ is the sequence defined before \eqref{eq:lambda} in Section \ref{ssc:pointwise},
so we know that $b^n\to b$ in $C_T\mathcal C^{-\beta}$, $b^n \in C_T \mathcal C^\gamma$ for all $\gamma \in \R$ and $b^n$ is bounded and Lipschitz. We denote by $X^n$ the (strong) solution to the SDE 
\begin{equation}\label{eq:Xn}
X_t^n = X_0 + \int_0^t b^n(s, X_s^n) \di s + W_t,
\end{equation}
where $X_0\sim\mu$.

\begin{theorem}\label{thm:tight}
  Let Assumptions \ref{ass:param-b} hold.
Let $(b^n)$ be a sequence in $C_T\mathcal C^{(-\beta)+}$ converging to $b $ in $C_T\mathcal C^{-\beta}$. Let $(X, \mathbb P)$ (respectively $(X^n, \mathbb P^n)$) be a solution  to the  MP with distributional drift $b$ (respectively $b^n$) and initial condition $\mu$.
Then the sequence $(X^n, \mathbb P^n)$  converges in law to  $(X, \mathbb P)$. In particular, if $b^n \in C_T\mathcal C^{0+}$  and $X^n$ is a strong solution of 
\begin{equation*}
X^n_t = X_0 + \int_0^t b^n(s, X^n_s) \di s + W_t,
\end{equation*}
then $X^n$ converges to $(X, \mathbb P)$ in law. 
\end{theorem}

\begin{proof} 
The proof is identical to that of \cite[Proposition 29]{flandoli_et.al14}. In particular Step 4 therein deals with the convergence in law of $Y^n$, which is the solution of SDE \eqref{eq:Yn}, and Step 5 with the convergence in law of $X^n$. 
Notice that the drift $b$ therein lives in a different 
space than ours (Bessel potential spaces instead of H\"older-Besov spaces), and the initial condition in \cite{flandoli_et.al14} is deterministic, but the setting is otherwise the same.  
The only tools used in Step 4 and 5 are the tightness of the sequence
of laws of $Y^n$,
which we proved in Lemma \ref{lm:Yntight}, and  the uniform convergence of $u^n\to u, \nabla u^n \to \nabla u$ and $\psi^n \to \psi$, see Section \ref{ssc:pointwise}. 
Finally  setting $X_t:= \psi(t, Y_t)$ for $t\in[0,T]$, then $(X, \mathbb P)$ is the unique solution to the martingale problem with distributional drift $b$ and initial condition $\mu$ by Theorem \ref{thm:mart},  
because $(Y, \mathbb P)$ is the unique solution to \eqref{eq:SDEY} with initial condition $Y_0 \sim \nu$ where  $\nu $ is the pushforward measure of $\mu$ through $\phi$.

It remains to prove the last claim of the theorem, which follows because  $X^n$ is also a solution to the MP with distributional drift $b^n$  by  Lemma \ref{lm:strongMP}, so the first part of the theorem can be applied.
\end{proof}

The first existence and uniqueness result is  for the solution to the MP with distributional drift $b$ and deterministic initial condition $X_0 =x$.  We will extend the result to any random variable in Theorem  \ref{thm:X0} below.

\begin{proposition}\label{pr:mart}
 The martingale problem with distributional drift $b$ and i.c.\ $\delta_x$, for $x\in\R^d$, admits existence and uniqueness according to Definition \ref{def:MP}. 
 \end{proposition}
 
 \begin{proof}
 Let $(X, \mathbb P)$ be a solution to the MP. 
Setting $Y_t = \phi (t, X_t)$  and $Y_0 = y:=\phi(0,x)$, by Item (i) of Theorem \ref{thm:mart} we have that 
 $(Y, \mathbb P)$ is a solution in law to \eqref{eq:SDEY}. By Remark  \ref{rm:SW} the solution $ (Y, \mathbb P)$ is unique,  hence the law of $X $ under $\mathbb P$ is uniquely determined.
  
   Existence follows from the fact that equation \eqref{eq:SDEY} with $Y_0=y$ has a solution in law, say $(Y, \mathbb P)$, again by Remark  \ref{rm:SW}. Then setting $X_t := \psi(t, Y_t)$ by Item (ii) of Theorem \ref{thm:mart} we know that $(X, \mathbb P)$ is a solution in law to MP with distributional drift $b$ and i.c.\ $\delta_x$.
\end{proof}

Next we  extend the existence and uniqueness result of Proposition  \ref{pr:mart} to the general case when the  initial condition $X_0$  is a random variable rather than a deterministic point.

\begin{theorem}\label{thm:X0}
Let Assumption \ref{ass:param-b} hold and let  $\mu$ be a probability measure on $\R^d$. Then there exists a unique solution $(X, \mathbb P)$ to the martingale problem with distributional drift $b$ and initial condition $\mu$.
\end{theorem}
\begin{proof}
\emph{Existence.} The idea is to use a {\em superposition} argument in order to glue together the solutions of MP with a deterministic initial condition $x$, for all possible initial conditions $x$.  This is implemented using the process $Y_t= \phi(t, X_t)$. 

We have the measure $\mu$ on $(\mathbb R^d,\mathcal B (\R^d) )$ which is the law of the initial condition $X_0$ and we define a new measure $\nu$ on the same space given by  
$\nu(B)= \mu(\psi(0, B))$ for any $B \in \mathcal B (\R^d) $. Notice that
$\nu$ is the pushforward  of $\mu$ through the function $\phi$, where $\psi= \phi^{-1}$ has been defined in \eqref{eq:psi}, thus $\nu$ plays the role of the initial condition for the process $\phi(t, X_t)$.
Let $Y$ be the canonical process and $\mathbb P^y$ be a law of the canonical process on $ \mathcal C_T $ such that $(Y,\mathbb P^y)$ is the unique weak solution to \eqref{eq:SDEY} with $Y_0 =y$. 
Then it is known by \cite[Theorem 7.1.6]{stroock_varadhan} that $ (y, C)\mapsto \mathbb P^y(C)$ is a random kernel for $y\in \R^d$ and $C\in \mathcal B(\mathcal C_T )$, hence the probability  $\mathbb P$ given by
\begin{equation}\label{eq:P}
\mathbb P(C):= \int \mathbb P^y (C)\nu(\di y)
\end{equation} is well-defined. 
Setting  $X_t := \psi(t, Y_t)$, our candidate solution to the MP with distributional drift $b$ and initial condition $\mu$ is $(X, \mathbb P) $. 
First we observe that for any $C \in \mathcal B(\mathcal C_T)$  of the form $C= \{\omega: \omega_0 \in B\} $  with some $B\in \mathcal B(\R^d)$, we have
\begin{equation}\label{eq:1B}
\mathbb P^y(C) = \mathbb P^y (\omega\in C) =   \mathbb P^y (Y_0 \in B) = \mathbbm {1}_{B}(y),
\end{equation}
having used that  $\mathbb P^y$-a.s.\ the canonical process $Y$ is such that $Y_0=y$. 
This will allow us to show that the initial condition $X_0$ has law $\mu$. Indeed, for any $A\in \mathcal B(\R^d)$ we set $B = \phi(0, A)$  and  we calculate
\begin{align}\label{eq:ic1}
\mathbb P(X_0 \in A) &=  \mathbb P(\psi(0, Y_0) \in A)
= \mathbb P(Y_0  \in \phi(0,A))
= \mathbb P(Y_0  \in B).
\end{align} 
Now using the definition \eqref{eq:P} of $\mathbb P$  and setting     $C= \{Y_0 \in B\} $ we have $\mathbb P(Y_0  \in B)= \mathbb P(C) = \int \mathbb P^y (C) \nu(\di y) $ and by \eqref{eq:1B} we have
\begin{equation}\label{eq:ic2}
\mathbb P(Y_0  \in B) =  \int_{B} \nu(\di y)=  \nu (B)= \nu ( \phi(0, A)).
\end{equation}
Finally using the definition of $\nu$  and the fact that  $ \psi$ is the inverse of $\phi$  we have $\mathbb P(X_0 \in A)= \mathbb P(C)=\mu (A)$ as wanted.

Next we show that for every $f\in \mathcal D_{\mathcal L}$ the process  
 \begin{equation}\label{eq:MfX}
 M^{f}_u (X): = f(u, X_u) - f(0, X_0) - \int_0^u (\mathcal L f)(r, X_r) \di r,
 \end{equation} 
is a martingale under $\mathbb P$, that is for every $f\in \mathcal D_{\mathcal L}$ and $F_s$ bounded and continuous functional on $\mathcal C_s$ (see Section \ref{ssc:prob-notation}) we have
\begin{equation*}\label{eq:equivMP}
\mathbb E [M^{f}_t(X) F_s(X)] =  \mathbb E [M^{f}_s (X) F_s(X)],
\end{equation*}
for all $0\leq s\leq t\leq T$. Indeed we notice that under $\mathbb P^y$ we have $Y_0\sim \delta_y$ hence $X_0\sim \delta_{\psi(0, y)}=:\delta_x$. 
Moreover $(Y,\mathbb P^y)$ is a solution of  \eqref{eq:SDEY} with i.c. $Y_0=y$, hence by Theorem \ref{thm:mart} part (ii) we have that $(X_\cdot:= \psi(\cdot, Y_\cdot),\mathbb P^y)$ is a solution to the MP with distributional drift $b$ and i.c. $X_0 \sim \delta_x$ thus by the definition of $\mathbb P$ given in  \eqref{eq:P}  we get
\begin{align*}
\mathbb E [(M^{f}_t (X)-M^{f}_s(X)) F_s(X)] &= \int 
\mathbb E^{y} [(M^{f}_t(X) -M^{f}_s(X)) F_s(X)] \nu (\di y) =  0,
\end{align*}
where we denoted by $\mathbb E^y$ the expectation under $\mathbb P^y$.

{\em Uniqueness.} Here the idea is to use {\em disintegration} in order to reduce the MP to MPs with deterministic initial condition. We proceed by stating and proving two preliminary facts.

\begin{description}
\item [\bf Fact 1]  Let  $E^1$ be a dense countable set in  $C_c(\R)$,  $E^2$   be a dense countable set  in $\mathcal D_{\mathcal L}$ and $E^{ \mathcal C_s}$ be a countable set of bounded continuous functionals such that for every bounded continuous functional $F_s \in \mathcal C_s$ there exists a sequence $(F_s^n)\subset E^{ \mathcal C_s}$ such that $F^n_s \to F_s$ in a pointwise uniformly bounded way, see \eqref{eq:contFsep}. 
 A couple $(X, \mathbb P)$ is a solution to the MP with
  distributional drift $b$ and initial condition $X_0$ if and only if 
\begin{equation}\label{eq:Mf}
\mathbb E [M^{f}_t(X) F_s(X) g(X_0)] =  \mathbb E [M^{f}_s(X) F_s(X)g(X_0)],
\end{equation}
 for every $f\in E^2, F_s\in E^{\mathcal C_s}, g\in E^1$ and $s<t$ with $s,t \in \mathbb Q \cap [0,T]$, where    $ M^{f}_u (X)$ is given by \eqref{eq:MfX}.

 This fact  can be seen as follows. First we notice that, since $M^f$ are bounded processes, if $M^f$ is a local martingale, then it is also a true martingale, and hence  the MP with distributional drift is equivalent to  \eqref{eq:Mf} for all $f\in \mathcal D_{\mathcal L}$, $F_s\in \mathcal C_s$ and $g\in C_c$ and $s<t$ with $s,t \in [0,T]$.  
 
 Next one can show that this is equivalent when choosing    $s<t,
 s,t \in \mathbb Q \cap [0,T]$. Indeed for any bounded and continuous functional $F_s$  on $ \mathcal C_s$,
 for a sequence  of rational times $s_n \downarrow s$ with $s<s_n<t$,
 we can associate a sequence of bounded and continuous  functionals $F_{s_n} $ on $\mathcal C_{s_n}$ by setting $F_{s_n}(\eta):= F_s (\eta\vert_{[0,s]})$ for $\eta \in \mathcal C_{s_n}$.
 This allows to replace the condition $s\in \mathbb Q\cap [0,T]$ with $s\in [0,T]$.  In order to replace $t \in \mathbb Q\cap [0,T]$ with $t\in [0,T]$ we choose a rational sequence $t_n \in (s, T] $ such that $t_n \to t$ and use the fact that the local martingale $t \mapsto M^f_t$ is continuous and  Lebesgue dominated convergence theorem.
  
Finally we use again
 Lebesgue dominated convergence theorem to see  the validity of \eqref{eq:Mf}  for all $f\in \mathcal D_{\mathcal L}$, $F_s\in \mathcal C_s$ and $g\in C_c$  and $s<t$ with $s,t \in \mathbb Q \cap [0,T]$. 
 
  We remark that   $E^1$ exists because $C_c$ is separable by  \cite[Lemma 5.7 (ii)]{issoglio_russoPDEa},  $E^2$ exists because  $\mathcal D_{\mathcal L}$ is separable by Proposition \ref{pr:DLsep} and  $E^{\mathcal C_s}$ exists by Lemma \ref{lm:contFsep}, whose statement and proof has been postponed at the end of this section.
 
\item [\bf Fact 2] Let $(X, \mathbb P)$ be a solution to the MP with
  distributional drift $b$ and i.c.\ $\mu$. There exists a random kernel $\mathbb P^{x}$ such that $\mathbb P = \int \mathbb P^{x} \di \mu (x)$,  where for $\mu$-almost all $x\in \R^d$, $\mathbb P^{x}$ lives on $\{\omega\in \Omega: X_0(\omega)=x\}$ and for any bounded and continuous functional $G: C[0,T] \to \R$ we have
\begin{equation}\label{eq:EiFX}
\mathbb E (G(X))  
= \int_{\R^d} \mathbb E^{x}( G(X)) \di \mu (x),
\end{equation} 
where $\mathbb E$ and $\mathbb E^{x}$ stand for the expectation under $\mathbb P$ and  $\mathbb P^{x}$ respectively.  

This follows from the disintegration theorem in  \cite[Chapter III, nos. 70--72]{dellacherie}. 
\end{description}

We now proceed with the proof of uniqueness. 
 Let $(X^1, \mathbb P_1)$ and $(X^2, \mathbb P_2)$ be two solutions to the MP with distributional drift $b$ and initial condition $X_0 \sim \mu$. Without loss of generality we can suppose that $X^1 = X^2 = X$ is the canonical process on $\Omega = \mathcal C_T$. 
Since $(X^i, \mathbb P_i), i=1,2$ is a solution of the MP, then by Fact 1 we have 
\[
\mathbb E_i [(M^{f}_t(X) -M^{f }_s(X))F_s(X)g(X_0)] = 0,
\]
 for all $0\leq s\leq t\leq T$, $s,t \in \mathbb Q$, $g \in E^1, f \in E^2$ and $ F_s \in E^{\mathcal C_s}$ and $i=1,2$.
We now apply Fact 2 to both $\mathbb P_1$ and $\mathbb P_2$, and in particular \eqref{eq:EiFX} with  $G (\eta) = (M^f_t(\eta) -M^f_t(\eta)) F_s(\eta) g(\eta_0) $ to rewrite the above equality as 
\begin{equation}\label{eq:Ei}
\int_{\R^d} \mathbb E_i^{x} [(M^{f}_t(X) -M_s^f(X))  F_s(  X) g(X_0)] \di \mu (x) =0,
\end{equation}
for all $0\leq s\leq t\leq T$, $s,t \in \mathbb Q$, $g \in E^1, f \in E^2$ and $ F_s \in E^{\mathcal C_s}$ and $i=1,2$.
 Now we recall that for $\mu$-almost all $x$, we have  $ X_0(\omega)=x$,  $\mathbb P^x_i$-a.s., thus equation \eqref{eq:Ei} becomes 
\[
\int_{\R^d} g(x) \mathbb E_i^{x} [(M^{f}_t(X) -M_s^f(X))  F_s(  X)] \di \mu (x) =0,
\]
for every $0\leq s\leq t\leq T$, $s,t \in \mathbb Q$, $g \in E^1, f \in E^2$ and $ F_s \in E^{\mathcal C_s}$ and $i=1,2$.    Since  $g$ is arbitrarily  chosen in a dense set of $C_c(\R)$ then we have
\begin{equation}\label{eq:E1E2}
 \mathbb E_i^{x} [(M^{f}_t(X) -M_s^f(X))  F_s(  X)] =0 \quad \mu\text{-a.e.},
\end{equation} 
for every $0\leq s\leq t\leq T$, $s,t \in \mathbb Q$, $f \in E^2$ and $ F_s \in E^{\mathcal C_s}$ and $i=1,2$. Note that \eqref{eq:E1E2} is true because the sets  $ \mathbb Q \cap [0,T],   E^2 $ and $ E^{\mathcal C_s}$ are countable. 
By Fact 1 this means that  the couple $(X,\mathbb P^{x}_i)$ is a solution to the MP with distributional drift $b$ and initial condition $\delta_x$, for $i=1,2$ for $\mu$-almost all $x$. 
By Proposition \ref{pr:mart}  we have uniqueness of the MP with deterministic initial condition  $\delta_x$   hence    for $\mu$-almost all $x$ we have   $ \mathbb P_1^x = \mathbb P_2^x$. Thus recalling the disintegration $\mathbb P_i = \int \mathbb P_i^{x} \di \mu (x)$ for $i=1,2$ from Fact 2,  we conclude $\mathbb P_1=\mathbb P_2 $ as wanted.
\end{proof}

We conclude the section with the proof of a technical result used in  Fact 1 in the proof of Theorem \ref{thm:X0}.

\begin{lemma}\label{lm:contFsep}
There exists a countable family $ D$ of bounded and continuous functionals from $ C([0,T];\R^d) $ to $\R$ such that any  bounded and continuous functional  $F: C([0,T];\R^d) \to \R$ can be approximated by a sequence $(F_n)\subset D$ in a pointwise uniformly bounded way, that is  
\begin{align}
\nonumber
&F_n \to F \text{ pointwise}\\
&\sup_n \sup_{\eta\in C([0,T];\R^d)} | F_n(\eta)| < \infty. \label{eq:contFsep}
\end{align}
\end{lemma}
\begin{proof}
We set $T=1$ without loss of generality. Let $\eta\in C([0,1];\R^d)$.
By \cite[Lemma 5.5]{issoglio_russoPDEa} we know that the function $t\mapsto F(\eta(t))$ can be approximated  by $F_n(\eta(\cdot)):= F(B_n (\eta, \cdot))$,
 where $(B_n)$ are the $\R^d$-valued Bernstein polynomials defined for any function  $\eta\in C([0,1];\R^d)$ by 
\[
B_n(\eta,t) : = \sum_{j=0}^n \eta(\frac jn ) t^j (1-t)^{n-j}{ n \choose j}.
 \]
Notice that the convergence is uniform in $t$. 
Now for fixed $n$ and $y_0,  y_1, \ldots, y_n \in \R^d$ we consider the function $f$ on $\R^{(n+1)d}$ defined by 
\[
f(y_0, y_1, \ldots, y_n):= F\left( \sum_{j=0}^n y_j (\cdot)^j (1-\cdot)^{n-j}{ n \choose j}\right),
\]  
so that $F_n(\eta) = f (\eta(\frac0n), \eta(\frac1n), \ldots, \eta(\frac nn))$. Notice that $\sup_{\eta\in C[0,T]} | F_n(\eta)| \leq \|F\|_\infty$.
We have thus reduced the problem to approximating any continuous bounded function  $f: \R^{(n+1)d} \to \R$. 
We further reduce the problem to continuous functions on $[-M, M]^{(n+1)d}$ by restriction, for some $M>0$. Indeed, a function  $f:[-M, M]^{(n+1)d} \to \R$ can be naturally extended to a bounded continuous function $\hat f$ on $\mathbb R^{(n+1)d}$ by setting for $x\in \R^{(n+1)d}$
\[
\hat f(x) = f(x_1\vee( -M) \wedge M, \ldots, x_{(n+1)d}\vee( -M) \wedge M).
\]
One can see that $ C([-M, M]^{(n+1)d})$ is separable by Stone-Weierstrass theorem. We denote by $D_{n, \text{fin}}$ the dense set in the set of bounded and continuous functions from $\R^{(n+1)d}\to \R$. 

The proof is concluded by setting  $D:= \cup_{n\in \N} D_n$, where 
\[
\begin{aligned}
D_n := \bigg\{ F: C([0,1];\R^d) \to \R:   F(\eta) &= f (\eta(\tfrac0n), \eta(\tfrac1n), \ldots, \eta(\tfrac nn) ), \\
& \eta \in C([0,1];\R^d) \text{ for some } f \in D_{n, \text{fin}} \bigg\},
\end{aligned}
\] 
which is a countable set of bounded functions. Then for any bounded and continuous functional $F :C([0,1];\R^d) \to \R$ we construct the sequence $(F_n)$ that converges to $F$ pointwisely  by choosing the appropriate  element $F_n\in D_n$. Since the convergence in $ C([-M, M]^{(n+1)d})$ is uniform we also have $\sup_n \sup_{\eta\in C([0,1];\R^d)} | F_n(\eta)| < \infty.$
\end{proof}

\begin{remark}
One could also define the domain $\mathcal D_{\mathcal L}$ of the martingale problem as a subset of the smaller space $C_T\mathcal C^{(2-\beta)-}$ instead of the larger space $C_T\mathcal C^{(1+\beta)+}$. On the other hand, one could  enlarge the domain by choosing functions with linear growth, namely in  $ C_TD\mathcal C^{\beta+}$. In both cases the analysis of the resulting MP is similar and  should lead to an equivalent  problem to the one studied in the present paper. We leave these details to the interested reader.
\end{remark}

\section{The solution of the MP as weak Dirichlet process}\label{sc:weakDir}

In this section we focus on the weak Dirichlet decomposition property
of the solution of the MP, which will be useful in Section \ref{sc:generalisedSDE} to characterize it as a solution of a generalized SDE.
We notice that a solution to the martingale problem with distributional drift $b$ is not a semimartingale  in general.
Indeed already in the fully studied case of dimension $d=1$, see  \cite[Corollary 5.11]{frw1},
one sees that the solution is a semimartingale if and only if $b$ is a Radon measure. We  can however discuss and investigate  other properties of this process,   which turns out to be  a weak Dirichlet process, and we identify the martingale component of the weak Dirichlet decomposition.

We start with the definition of weak Dirichlet process, that can be found in \cite{gozzi_russo06}, see also \cite{er, er2}. 
\begin{definition} \label{D51}
Let $X$ be a continuous  stochastic process on some probability space $(\Omega, \mathcal F, \mathbb P)$ and let $\mathcal F^X$ denote its canonical filtration.
\begin{itemize}
\item A process $\mathscr A$ is said to be an {\em $\mathcal F^X$-martingale orthogonal process} if $[N,\mathscr A]=0$ for every $\mathcal F^X$-continuous local martingale $N$. 
\item 
  The process $X$  is said {\em $\mathcal F^X$-weak Dirichlet} if it is the sum of an $\mathcal F^X$-local martingale $M$ and an $\mathcal F^X$-martingale orthogonal process $\mathscr A$.
  
When $\mathscr A_0=0$ a.s., we call $
X = M+ \mathscr A$  the standard decomposition. 
\end{itemize}
\end{definition}

\begin{remark}\ 
\begin{itemize}
\item The two equalities in the statement
  of Definition \ref{D51}, that is $[N,\mathscr A]=0$
 and  $
X = M+ \mathscr A$,  
  are meant up to indistinguishability with respect to~$\mathbb P$. 
\item The standard decomposition of a $\mathcal F^X$-weak Dirichlet process is unique.
\end{itemize}
\end{remark}



In the remainder of the section, we let $(X, \mathbb P)$ be the solution to the martingale problem with distributional drift $b$ and initial condition $\mu$, with $\mathbb P$ being a probability measure on some measurable space $(\Omega, \mathcal F)$ that will be fixed throughout. We  will make use of the space of processes $\mathscr C$, introduced in Section \ref{ssc:prob-notation}.
 Let Assumption \ref{ass:param-b} hold.

\begin{proposition}\label{pr:weakDirichlet}
Let $f\in C^{0,1}([0,T]\times \R^d)$. 
Then $f(t, X_t)$ is an $\mathcal F^X$-weak Dirichlet process. 
In particular $X$ is an $\mathcal F^X$-weak Dirichlet process.  \end{proposition}
\begin{proof}
We recall that by Theorem \ref{thm:mart} $X=\psi(t, Y_t)$ where  $\psi \in C^{0,1}$ and $(Y_t)$ is an $\mathcal F^X$-semimartingale.  Then $f(t, X_t) = (f\circ \psi)(t, Y_t)$ is a $C^{0,1}$ function of a semimartingale, hence it is a weak Dirichlet process by \cite[Corollary 3.11]{gozzi_russo06}.
\end{proof}

From now on we denote by   $f(t, X_t) = M^f + \mathscr A^f_t$ the standard decomposition of the weak Dirichlet process  $f(t, X_t) $ for $f\in C^{0,1}$.

In what follows we compute the covariation process between two martingale parts $M^f$ and $M^h$, for two functions $f,h\in C^{0,1}$. To do so we first need some preparatory lemmata dealing with functions in some subspace of  $\mathcal D_{\mathcal L} $.
We denote by $\mathcal D_{\mathcal L}^s$ the space given by
\begin{equation}\label{eq:DLs}
\mathcal D_{\mathcal L}^s:=\{ f \text{ such that } \exists \tilde f \in C^{1,2}_c \text{ and  } f = \tilde f \circ \phi\},
\end{equation}
which is obviously and algebra. Moreover  it is  a linear  subspace of $\mathcal D_{\mathcal L}$ by Lemma \ref{lm:DL}.

\begin{proposition}\label{pr:algebra}
For $f,h \in \mathcal D_{\mathcal L}^s $ we have 
\begin{equation}\label{eq:algebra}
\mathcal L (fh) = (\mathcal L f)h +( \mathcal L h) f + \nabla f \nabla h . 
\end{equation}
\end{proposition}
\begin{proof}
 Let $f,h \in \mathcal D_{\mathcal L}^s $ and let us compute the time derivative of the product $fh$. We have
\begin{align}\label{eq:dtfg}
\partial_t(fh) & = h\partial_t f+ f \partial_t h, 
\end{align}
which makes sense as we see below. Indeed, $h\partial_t f $ is well-defined because  $h\in C_T \mathcal C^{(1+\beta)+}$ and $
\partial_t f = \mathcal L f -\frac12 \Delta f  - \nabla f b $
is an element of $C_T\mathcal C^{(\beta-1)+}$. The latter holds  because
$ \mathcal L f \in C_T \mathcal  C^{0+}$, $\frac12 \Delta f \in C_T\mathcal C^{(\beta-1)+}$ and 
 $\nabla f b \in C_T\mathcal C^{-\beta}$, with $(\beta-1)\leq -\beta$. Similarly for  $ f \partial_t h$.

We also calculate the Laplacian of $fh$
  \begin{equation}\label{eq:deltafg}
\tfrac12 \Delta (fh) = \frac12 (h \Delta f + 2 \nabla f \nabla h + f\Delta h), 
 \end{equation}  
 where we recall that $\nabla f \nabla h:= \nabla f \cdot \nabla h$, and we calculate  the transport term
\begin{equation}\label{eq:nablafg}
b \nabla (fh) = b \nabla f \, h + b \nabla h \, f,
\end{equation}
which are well-defined by similar arguments. Collecting \eqref{eq:dtfg}, \eqref{eq:deltafg} and \eqref{eq:nablafg} then equality \eqref{eq:algebra} follows. 
\end{proof}

\begin{lemma}\label{lm:covariationDL}
Let $f,h \in \mathcal D_{\mathcal L}^s$. Then 
\begin{equation}\label{eq:covariationDL}
[M^f, M^h]_t = \int_0^t (\nabla f)(s, X_s)  (\nabla h)(s, X_s) \di s.
\end{equation}
\end{lemma}
\begin{proof}
By Proposition \ref{pr:algebra}, $fh\in\mathcal D_{\mathcal L}^s \subset \mathcal D_{\mathcal L},$ so  using the martingale problem, Proposition \ref{pr:weakDirichlet} (and considerations below) together with the uniqueness of the standard weak Dirichlet decomposition
 we have  
\begin{equation}\label{eq:fg1}
(fh)(t,X_t) = M^{fh}+ \int_0^t \mathcal L (fh) (s, X_s) \di s,
\end{equation}
{having incorporated the initial condition $(fh)(0,X_0)$ in the martingale part $M^{fh}$ so that $\mathscr A^{fh}_t = \int_0^t \mathcal L (fh) (s, X_s) \di s$ hence  $\mathscr A^{fh}_0=0$ as required.} It holds also
\begin{align}
\label{eq:semimartf}&f(t,X_t) = M_t^{f}+ \int_0^t \mathcal L f (s, X_s) \di s\\
\label{eq:semimartg}&h(t,X_t) = M_t^{h}+ \int_0^t \mathcal L h (s, X_s) \di s.
\end{align}
Integrating by parts 
$(fh)(t,X_t)$ and using \eqref{eq:semimartf} and \eqref{eq:semimartg}
 we have
\begin{align} \label{eq:fg2}
\notag
(fh)(t,X_t) =&  \int_0^t f(s, X_s) \di h(s, X_s) 
+ \int_0^t h(s, X_s) \di f(s, X_s) 
+ [f(\cdot, X), h(\cdot, X)]_t\\ 
=& \mathscr M_t +\int_0^t f(s, X_s) (\mathcal Lh)(s, X_s) \di s 
+ \int_0^t h(s, X_s) (\mathcal Lf)(s, X_s) \di s
+ [M^f, M^h]_t, 
\end{align}
 where $(\mathscr M_t)$ is some local martingale.  Equations \eqref{eq:fg1} and \eqref{eq:fg2} give two decompositions of the semimartingale $(fh)(t, X_t)$. By uniqueness of the decomposition and taking into account Proposition \ref{pr:algebra}, the conclusion \eqref{eq:covariationDL} follows. 
\end{proof}

\begin{remark}\label{rm:covariation}
We notice that both sides of \eqref{eq:covariationDL} are well-defined also for $f,h\in C^{0,1}$.
\end{remark}

\begin{lemma}\label{lm:density}
$\mathcal D_{\mathcal L}^s$ is dense in $C^{0,1}([0,T] \times \mathbb R^d)$.
\end{lemma}
\begin{proof}
 Let $\chi: \mathbb R \to \R_+$ be a smooth function such that 
\[
\chi (x) =\left\{
\begin{array}{ll}
0 \quad &x\geq0\\
1 \quad &x\leq-1\\
\in(0,1) \quad  & x\in(-1,0).
\end{array}
\right.
\]
We set  $\chi_n: \R^d \to \R$ as
$
\chi_n(x) : = \chi (|x| - (n+1)).
$
In particular 
\[
\chi_n (x) =\left\{
\begin{array}{ll}
0 \quad &|x|\geq n+1\\
1 \quad &|x|\leq n\\
\in(0,1) \quad  & \text{otherwise}.
\end{array}
\right.
\]
Let $f\in C^{0,1}$. Let us define $\tilde f:= f\circ \psi \in C^{0,1}$ and $\tilde f_n := \tilde f \chi_n$.
 Since $ \tilde f_n \to \tilde f$ in $C^{0,1}$ also $  f_n:= \tilde f_n \circ \phi \to  f$ in $C^{0,1}$, hence we reduce to the case where  $\tilde f =f\circ \psi$ has compact support. 

 We set  
\[
\tilde  f_m(t,x):= m \int_t^{t+\tfrac1m} (f\star \rho_m)(s,x) \di s,
\]
where $\rho_m$ is a sequence of mollifiers with compact support and $\star$ denotes the {space-}convolution. Then $\tilde  f_m \in C_c^{1, \infty } ([0,T] \times \R^d)$  and $\tilde  f_m \to \tilde f$ in $C^{0,1}$ hence $f_m:= \tilde f_m \circ \phi \to f$ in $C^{0,1}$.
\end{proof}

\begin{theorem}\label{thm:covariation}
Let $f,h \in C^{0,1}$. Then 
\begin{equation}\label{eq:covariation}
[M^f, M^h]_t = \int_0^t (\nabla f)(s, X_s)  (\nabla h)(s, X_s) \di s.
\end{equation}
\end{theorem}
\begin{proof}
First we notice that \eqref{eq:covariation} holds for every $f,h\in \mathcal D_{\mathcal L}^s$ by Lemma \ref{lm:covariationDL}. Each side of \eqref{eq:covariation} is well-defined for $f,h \in C^{0,1},$ by Remark \ref{rm:covariation}. Moreover by  Lemma \ref{lm:density} $\mathcal D_{\mathcal L}^s \subset C^{0,1}$ is a dense subspace. 

Next we show that, for fixed $h\in \mathcal D_{\mathcal L}^s$, the map $f\mapsto [M^f, M^h]$ is continuous and linear from $C^{0,1}$ to $\mathcal C$. For this we make use of Banach-Steinhaus theorem for F-spaces, see e.g.\ \cite[Theorem 2.1]{dunford-schwartz}. Indeed, the space $C^{0,1}$ is clearly an F-space, and so is the linear space of continuous processes $\mathscr C$ equipped with the
u.c.p.\ topology. 
Let $[M^f, M^h]^\varepsilon$ denote the $\varepsilon$-regularization of the bracket $[M^f, M^h]$, see
\cite[Definition 4.2]{Russo_Vallois_Book} or \cite[Section 1]{russo_vallois95}
for a precise definition. 
 Let $h\in \mathcal D_{\mathcal L}^s$
be fixed.
 The operator $T^\varepsilon: f \mapsto [M^f, M^h]^\varepsilon$ is linear and continuous from $C^{0,1}$ to $\mathscr C$. Finally $[M^f, M^h]$ is well-defined as a u.c.p.-limit of $[M^f, M^h]^\varepsilon$, see \cite[Proposition 1.1]{russo_vallois95}. Thus by Banach-Steinhaus the map $f\mapsto [M^f, M^h]$ is continuous from $C^{0,1}$. 
  Since both members of \eqref{eq:covariation} are continuous and linear, then \eqref{eq:covariation} extends to all $f\in C^{0,1}$ and $h\in \mathcal D_{\mathcal L}^s$.

 Finally let $f\in C^{0,1}$ be fixed. By the same reasoning as above we extend \eqref{eq:covariation} to $h\in C^{0,1}$.
\end{proof}

\begin{corollary}\label{cor:Acont}
The map $f\mapsto \mathscr A^f$ is continuous (and linear) from $C^{0,1}$ to $\mathscr C$.
\end{corollary}
\begin{proof}
Since  $f_n \to 0$ in $C^{0,1}$ then $f_n(\cdot , X)\to 0$ u.c.p. By Theorem \ref{thm:covariation}  $[M^{f_n}]\to 0$, and taking into account \cite[Chapter 1, Problem 5.25]{karatzasShreve} we have that $M^{f_n}\to 0$ u.c.p. Using the decomposition $f_n(\cdot, X) = M^{f_n}+ \mathscr A^{f_n}$ we have $\mathscr A^{f_n} \to 0$ u.c.p. and the proof is concluded. 
\end{proof}

\begin{remark}\label{rm:Mid}
Let $\text{id}_i (x)= x_i$. Then $\text{id}_i \in C^{0,1}$. Setting $M^{\text{id}} = (M^{\text{id}_1}, \ldots, M^{\text{id}_d})^\top$ then by  Theorem \ref{thm:covariation} we have 
 \[
[ M^{\text{id}_i},  M^{\text{id}_j}]_t = \delta_{ij}t. 
\] 
Hence  by L{\'e}vy characterization theorem this implies that $M^{\text{id}}- X_0$ is a standard $d$-dimensional Brownian motion. We denote this Brownian motion by $W^X$.
\end{remark}

\begin{proposition}\label{pr:Mf}
For $f\in C^{0,1}([0,T] \times \R^d)$ we have 
\[
M^f_t = f(0, X_0) + \int_0^t \nabla f(s, X_s) \cdot  \di M^{\text{id}}_s.
\]
\end{proposition}
\begin{proof}
Recall that  we write
\begin{equation}\label{eq:fMA}
f(t, X_t) = M^f_t + \mathscr A^f_t,  
\end{equation}
where the right-hand side is the standard (unique) decomposition
of the left-hand side,
as an $\mathcal F^X$-weak Dirichlet process. In particular $ \mathscr A^f$ is an $\mathcal F^X$-orthogonal process with $  \mathscr A^f_0=0$ and $M^f$ is the martingale component. We define $\tilde {\mathscr A}^f$ so that 
\[
f(t, X_t) =   f(0, X_0) + \int_0^t \nabla f(s, X_s) \cdot  \di M^{\text{id}}+ \tilde {\mathscr A}^f_t.
\]
We will prove later that 
\begin{equation}\label{eq:N}
[\tilde {\mathscr A}^f, N] =0 \text{ for all continuous local $\mathcal F^X$-martingales }N.
\end{equation}
From \eqref{eq:N} we have that $\tilde {\mathscr A}^f$ is an $\mathcal F^X$-martingale orthogonal process with  
$\tilde {\mathscr A}^f_0 =f(0, X_0)-f(0, X_0) =0$, 
thus by uniqueness of the decomposition of weak Dirichlet processes it must be $\tilde {\mathscr A}^f ={\mathscr A}^f$ and therefore  
\[
M^f_t = f(0, X_0) + \int_0^t \nabla f(s, X_s) \cdot \di M_s^{\text{id}},
\]
as wanted. It remains to prove \eqref{eq:N}. By definition of $\tilde {\mathscr A}^f$ and  \eqref{eq:fMA} we have 
\begin{align} \label{eq:NN} 
\nonumber
[\tilde {\mathscr A}^f, N]_t &= [f(\cdot, X_\cdot) , N]_t - [\int_0^\cdot \nabla f(s, X_s)  \cdot \di M^{\text{id}}, N]_s\\
& = [M^f, N]_t-  \int_0^t  \nabla f(s, X_s) \cdot \di[ M^{\text{id}}, N]_s,
\end{align} 
having used the weak Dirichlet decomposition $f(\cdot, X)= M^f +{\mathscr A}^f $, where ${\mathscr A}^f$ is an $\mathcal F^X$-martingale orthogonal process. Regarding  $N$, now we observe that by Kunita-Watanabe decomposition  there is an $\mathcal F^X $-progressively  measurable process $\xi$ and an orthogonal local martingale $O$ such that
\[
N_t = N_0 + \int_0^t \xi_s \cdot dM^{\text{id}}_s +O_t.
\]
Thus the covariation with $M^{\text{id}}$ gives
\begin{align*}
[ M^{\text{id}}, N]_t &= [ M^{\text{id}},  \int_0^\cdot \xi_s \cdot \di  M^{\text{id}}_s ]_t =\int_0^t \xi_s \di s  ,
\end{align*}
since $[ M^{\text{id}_i},  M^{\text{id}_j}]_t = \delta_{i,j} t $ by Remark \ref{rm:Mid}. 
We  calculate $[  M^f , N]_t$ using Theorem  \ref{thm:covariation} to get
\[
[  M^f , N]_t =  [ M^{f},  \int_0^\cdot \xi_s  \cdot \di  M^{\text{id}}_s ]_t =\int_0^t \xi_s \cdot  \di   [ M^{f}, M^{\text{id}}]_s   = \int_0^t \xi_s \cdot  \nabla f (s, X_s)  \di  s  .
\]
Plugging these two covariations into  \eqref{eq:NN} we get 
\[
[\tilde {\mathscr A}^f, N]_t = \int_0^t \xi_s \cdot \nabla f (s, X_s)  \di  s  - \int_0^t  \nabla f(s, X_s) \cdot   \xi_s \di s =0 ,
\]
 which is \eqref{eq:N} as wanted. 
\end{proof}

We conclude this section with some final remarks.
\begin{remark}\label{rm:A}
 \begin{itemize}
\item[(i)]  We recall that $ \mathcal D_{\mathcal L}^s  \subset \mathcal D_{\mathcal L} \subset C^{0,1}$. Thus  for $f\in \mathcal D_{\mathcal L}^s$  by uniqueness of the weak Dirichlet decomposition and by the martingale problem we have $\mathscr A^f_t = \int_0^t (\mathcal L  f)(s, X_s) \di s$. Therefore we have that $f\mapsto \mathscr A^f$ is the continuous linear extension of $f\mapsto \int_0^t (\mathcal L  f)(s, X_s) \di s$ taking values in
  $\mathscr C$.

\item[(ii)]  We recall that the function $\text{id}_i$ solves PDE \eqref{eq:PDE}  so we have  $\mathcal L \text{id}_i = b^i$, see Section \ref{ssc:pointwise}. Hence  taking $f=\text{id}_i$ for some $i\in\{1, \ldots, d \}$ one gets $X = M^{\text{id}_i} + \mathscr A^{\text{id}_i} $, where  formally
\[
\mathscr A^{\text{id}_i}  = ``\int_0^\cdot b^i(s, X_s) \di s",
\]
by the first point in this Remark. Putting all components together one would get indeed 
\[
\mathscr A^{\text{id}} :=(\mathscr A^{\text{id}_i} )_i = 
 ``\int_0^\cdot b(s, X_s) \di s".
\] 
Plugging  this  into the  decomposition $\text{id}(X_t) = M_t^{\text{id}} + \mathscr A_t^{\text{id}} $ and using  Remark \ref{rm:Mid} gives the (formal) writing 
\[ X_t  =X_0 + W^X_t+ ``\int_0^t b(s, X_s) \di s"\]
as expected. 
 Notice however  that in general $\text{id}_i \notin \mathcal D_{\mathcal L}$ since  $b\in C_T \mathcal C^{-\beta}$ so in general $ b \notin C_T \bar{\mathcal C}_c^{0+}$. This is why the writing above is only formal. We will introduce an extended domain in the next section to make this argument rigorous.  
\end{itemize} 
\end{remark}

 \section{Generalised SDEs and their relationship with MP} \label{sc:generalisedSDE}

 In this final section we investigate the dynamics of the process $X$ which formally solves the SDE $\di X_t = b(t, X_t) \di t + \di W_t$ and compare it to the solution to the martingale problem. First we define a notion of solution for the formal SDE, a definition that amongst other things involves weak Dirichlet processes. We show that any solution to the MP is also a solution of the formal SDE and a chain rule holds (Theorem \ref{thm:cr}). Finally we {\em close the circle} by showing that, under the stronger assumption for $X$ to be
a Dirichlet process, $X$ being a solution
to the formal SDE is equivalent to being a solution to the MP (Corollary \ref{cor:Xsoliff}). We recall that $X$ is an ${\mathcal F}^X$-Dirichlet process
  if it is the sum of an ${\mathcal F}^X$-local martingale
    plus an adapted zero quadratic variation process. In this section there is always  an underlying measurable space  $(\Omega, \mathcal F)$. 
    
We make a further technical assumption on the support of the singular drift $b$. This Assumption is a standing assumption until the end of the paper.
\begin{assumption}\label{ass:param-bc}
 Let $b \in  C_T \bar{\mathcal C}_c^{(-\beta)+} $.
\end{assumption}

As mentioned above, 
the idea of the current section is inspired by Remark \ref{rm:A} and consists in  further investigating to which extent our solution to the martingale problem is the solution of an SDE of the form
\[
 X_t  = X_0 + W^X_t+ ``\int_0^t b(s, X_s) \di s",
\]
where $X_0 \sim \mu$.
We note that if $b=l$ were a function, the interpretation of $``\int_0^t l(s, X_s) \di s"$ would indeed be the integral $\int_0^t l(s, X_s) \di s$.  In particular, $\int_0^t l(s, X_s) \di s$ is well-defined for any $l \in  C_T \bar{ \mathcal C}_c^{0+}$.
 We will study various properties of $l \mapsto \int_0^t l(s, X_s) \di s$ for a reasonable class of distributions $l$ (which includes for example $b\in  C_T \bar{ \mathcal C}_c^{(-\beta)+}$ from Assumption \ref{ass:param-bc}), proceeding similarly to \cite{russo_trutnau07}.

\begin{definition}
Let $\mathbb P$ be a probability measure on $(\Omega, \mathcal F)$.
We say that a process $X$ fulfills  the {\em local time property} with respect to a topological vector space $B \supset  C_T \bar{ \mathcal C}_c^{0+}$ if  $  C_T \bar{ \mathcal C}_c^{0+}$ is dense  in $B $ and the map from $  C_T \bar{ \mathcal C}_c^{0+}$ with values in $ \mathscr C$ defined by 
\[
 l \mapsto \int_0^t l(s, X_s) \di s
\]
admits a continuous extension to $B$ (or equivalently it is continuous with respect to {the topology of} $B$) which we denote by $A^{X, B}$.
\end{definition}

Notice that this notion has been first defined in a different context in \cite[Definition 6.1]{russo_trutnau07}, see also \cite[Remark 6.2]{russo_trutnau07} for the links to local time.
Using the  local time property we now introduce a notion of solution to SDE  which is different from the martingale problem.  We will then study its properties and links to the solution to the martingale problem. 

\begin{definition}\label{def:Bsol}
Let $\mathbb P$ be a probability measure on $(\Omega, \mathcal F)$.
Given $b\in B\subset \mathcal S'(\R^d)$ we say that $X$ is a {\em $B$-solution} to 
\[
 X_t  =X_0 + W_t+ \int_0^t b(t, X_t) \di s,
\]
if there exists a  Brownian motion  $W= W^X$  and 
\begin{itemize}
\item[(a)] $X$ fulfills the  local time property with respect to $B$; 
\item[(b)] $b\in B$;
\item[(c)] $ X_t  = X_0 + W^X_t+ A_t^{ X,B}(b)$;
\item[(d)] $X$ is an $\mathcal F^X$-weak Dirichlet process.
\end{itemize}
\end{definition}


\begin{remark}
Some examples of $B$  are $B=  C_T \bar{ \mathcal C}_c^{0+}$ and  $B= C_T \bar{ \mathcal C}_c^{(-\beta)+}$. Indeed,  $  \bar{ \mathcal C}_c^{0+}$ is dense in  $\bar{ \mathcal C}_c^{(-\beta)+}$ since by \cite[Lemma 5.4 (i)]{issoglio_russoPDEa} $\mathcal S \subset\bar{ \mathcal C}_c^{0+} $ and $\mathcal S$ is dense in $\bar{ \mathcal C}_c^{(-\beta)+}$. Finally by \cite[Remark B.1]{issoglio_russoMK} we conclude that 
$ C_T \bar{ \mathcal C}_c^{0+}$ is dense in  $C_T \bar{ \mathcal C}_c^{(-\beta)+}$.
 \end{remark}

 Below we will investigate $B$-solutions for $B=  C_T \bar{\mathcal C}^{(-\beta)+}_c$. We denote by 
\begin{equation}\label{eq:DLB}
\mathcal D_{\mathcal L}^B:= \left \{f \in \mathcal D^0_{\mathcal L} \text{ such that } g:= \mathcal L f \in B 
\right \}.
\end{equation}

\begin{remark}\label{rm:b}
Let  $B =  C_T \bar{ \mathcal C}_c^{(-\beta)+}$.
Notice that  $f=\text{id} \in \mathcal D_{\mathcal L}^B$ and $\mathcal L \, id=b$, in the sense that $\mathcal L \, \text{id}_i=b^i$ for all $i=1, \ldots, d$ as recalled in Remark \ref{rm:A} item (ii). 
\end{remark}

\begin{theorem}\label{thm:cr}
Let $B =  C_T \bar{ \mathcal C}_c^{(-\beta)+}$.
Let $(X, \mathbb P) $ be the solution to the martingale problem with distributional drift $b$ and i.c.\ $\mu$. Then there exists a Brownian motion $W^X$ with respect to $ \mathbb P $ such that  $X$ is a $B$-solution of 
\[
 X_t  = X_0 + W^X_t+ \int_0^t b(s, X_s) \di s,
\]
where $X_0 \sim \mu$.
Moreover, for every $f\in \mathcal D_{\mathcal L}^B $ we have  {the chain rule}
\begin{equation}\label{eq:cr}
f(t, X_t) = f(0,X_0) + \int_0^t (\nabla f ) (s, X_s) \cdot \di W^X_s + A^{X,B}_t (\mathcal L f),
\end{equation}
and the equality 
\begin{equation}\label{eq:AA}
A^{X,B}_t (\mathcal L f) = \mathscr A^f_t.
\end{equation} 
\end{theorem}
\begin{remark}\label{rm:cr}
Notice that point (c) in Definition  \ref{def:Bsol} provides the standard decomposition of the weak Dirichlet process $X$, where the local martingale component is given by $X_0+W^X$ and the martingale orthogonal process is given by  $A^{X,B}_t (b) = \mathscr A^{\text{id}}_t$ in view of \eqref{eq:AA} and Remark \ref{rm:b}.
\end{remark}

\begin{proof}[Proof of Theorem \ref{thm:cr}]
For ease of notation we write $A^X$ in place of $A^{X,B}$.

Let $(X, \mathbb P)$ be the solution  to the martingale problem with distributional drift $b$ and i.c.\ $\mu$. We have to show that the four conditions of Definition \ref{def:Bsol} are satisfied. Clearly $b\in B$ which is point (b) of   Definition \ref{def:Bsol}. 
By Proposition \ref{pr:weakDirichlet},  for every $f\in C^{0,1}$ we have that $f(t, X_t)$ is an $\mathcal F^X$-weak Dirichlet process,  hence  $X$ is also a weak Dirichlet process (point (d) of Definition  \ref{def:Bsol}) 
with decomposition 
\[
f(t, X_t) =  M^f_t + \mathscr A^f_t . 
\]
 Next we check the   local time property, which is point (a) of   Definition \ref{def:Bsol}. We use that $X$ solves the martingale problem for every $f\in \mathcal D_{\mathcal L} \subset C^{0,1}$ (thus $f(t, X_t)- \int_0^t (\mathcal L f) (s, X_s) \di s $ is a local martingale) and uniqueness of the weak Dirichlet decomposition to get
\begin{equation}\label{eq:A}
\mathscr A^f = \int_0^\cdot (\mathcal L f) (s, X_s) \di s =  A^X( \mathcal L f),
\end{equation}
where the second equality holds because  $ \mathcal L f\in C_T \bar{\mathcal C}_c^{0+}$. We want to show that $  A^X$ extends to all $g\in B = C_T \bar{ \mathcal C}_c^{(-\beta)+}$. 
Let us denote by $T$ the  map 
\[
\begin{array}{llll}
T:& C_T \bar{\mathcal C}_c^{(-\beta)+} &\to& C_T D \mathcal C^{\beta+} \\
& g& \mapsto & T(g) := v,
\end{array}
\]
 where $v$ is the unique solution in $C_T \mathcal C^{(1+\beta)+}$  of PDE 
 $$
\left\{
\begin{array}{l}
 \mathcal L v = g\\
 v(T) = 0,
\end{array}
\right.
 $$
which is PDE \eqref{eq:PDE} with $v_T =  0$, see Section \ref{ssc:pointwise}.   It is clear that for $f\in \mathcal D_{\mathcal L}$ and $g= \mathcal L f \in C_T \bar{\mathcal C}_c^{0+}$ 
 we have $T(g)=f$ so that \eqref{eq:AA} writes
\[
\mathscr A^{T(g)} =  A^X(g).
\] 
Now we recall  that  $g\mapsto T(g)\in   C_T  \mathcal C^{(1+\beta)+}  \subset C^{0,1}$ is continuous, see Section \ref{ssc:pointwise}, in particular when $g_n\to g$ in $C_T \bar{\mathcal C}_c^{-\beta} $ then $f_n =T(g_n)\to T(g)= f$ in $C_T \mathcal C^{(1+\beta)+}  \subset C^{0,1} $. Moreover  by Corollary \ref{cor:Acont} also the map $f \mapsto \mathscr A^f$ is continuous from $C^{0,1}$ to $\mathscr C$.
Now we use the density of $ C_T \bar{\mathcal C}_c^{0+}$ in $C_T \bar{\mathcal C}_c^{(-\beta)+}$ to conclude that the  local time property holds  and also \eqref{eq:AA} holds.
Point (c) in Definition \ref{def:Bsol} follows from the chain rule \eqref{eq:cr} (shown below) for $f= \text{id}$ using Remark \ref{rm:b}.

It is left to prove that the chain rule \eqref{eq:cr} holds.  
We define $W^X:= M^{\text{id}} - X_0$, which is a Brownian motion 
by Remark \ref{rm:Mid}. 
First we prove that \eqref{eq:cr} holds for $f \in \mathcal D_{\mathcal L}$. Indeed by Proposition \ref{pr:Mf} we know that  $M^f_t = f(0, X_0) + \int_0^t (\nabla f )(s, X_s) \cdot \di W_s^X$
so using that $X$ is a solution to  the martingale problem we easily get that \eqref{eq:cr} holds for  $f\in \mathcal D_{\mathcal L} $. 
In order to extend it to $f\in \mathcal D_{\mathcal L}^B$, we use the operator $T$ and rewrite the chain rule \eqref{eq:cr} as
\begin{equation}\label{eq:crg}
(Tg)(t, X_t) - (Tg)(0, X_0) -  \int_0^t \nabla( Tg) (s, X_s) \cdot \di W^X_s =  A_t^X(g),
\end{equation}
for all $g\in B = C_T \bar{\mathcal C}_c^{(-\beta)+}$. 
Notice that \eqref{eq:crg} holds for $g\in C_T \bar{\mathcal C}_c^{0+}$ since 
\eqref{eq:cr} holds for $f\in \mathcal D_{\mathcal L}$ with $\mathcal Lf =g$.
The left-hand side of \eqref{eq:crg} is continuous from $B$ to $\mathscr C$ because it is the composition of continuous operators. 
The right-hand side  of \eqref{eq:crg} extends from $g \in C_T \bar{\mathcal C}_c^{0+}$ to $g\in B$ by the  local time property (a). 
Since $  C_T \bar{\mathcal C}_c^{0+}$ is dense in $B$ then \eqref{eq:crg} extends to $B$, which is  \eqref{eq:cr} as wanted.
\end{proof}

\begin{remark}
Notice that if, in the previous  proof, we defined the solution operator $T$ using a different terminal condition $v_T \in \mathcal C^{(1+\beta)+} $, $v_T\neq 0$, it  would have led to the same operator $A^{X,B}$. This can be seen by noticing that the operator is the unique extension of the integral operator $l \mapsto \int_0^t l(s, X_s) \di s$.  
\end{remark}

We now introduce a refined notion of $B$-solution, which will be used later.
\begin{definition}\label{def:rlt}
Let $\mathbb P$ be a probability measure on $(\Omega, \mathcal F)$.
  We say that $X$  is a {\em reinforced $B$-solution} of  
\[
 X_t  = X_0 + W_t^X+ \int_0^t b(t, X_t) \di s
\]
if 
\begin{itemize}
\item[(i)]
it is a $B$-solution of the  SDE in the sense of Definition \ref{def:Bsol}; 
\item[(ii)]
 for any $f\in C_b^{1,2,B}$, where
 \[
 C_b^{1,2,B}:=\{ f\in C_b^{1,2} \text{ such that } \dot f +\frac12 \Delta f  \in C_T \bar {\mathcal C}_c^{0+}  \text{  and } \nabla f b \in B  \},
\]
  then
\begin{equation}\label{eq:rlt}
\int_0^t (\nabla f )(s, X_s) \cdot \di^- A_s^{X,B} ( b) =  A^{X,B}_t (\nabla f \, b), 
\end{equation}
where the forward integral $d^-A$ is the one given in \cite{russo_vallois93} in the one-dimensional case, which can be straightforwardly extended
to the vector case.  In particular, for a locally bounded integrand process $Y$ and
a continuous integrator process $X$ we denote
$$ \int_0^t Y_s \cdot \di^- X_s = \sum_{i =1}^d \int_0^t Y^i_s  \di^- X^i_s.$$
\end{itemize}
\end{definition}

\begin{remark}\label{rm:B}
\begin{itemize}
\item[(i)] When $b \in C_T\bar {\mathcal C}_c^{0+}  $  and $ f\in C^{1,2}_b$ then $\nabla f b \in C_T \bar {\mathcal C}_c^{0+}$ because  we can choose the approximating sequence $b_n \to b$ with compact support to construct the approximating sequence $\nabla f b_n \to \nabla f b$. In this case  equality \eqref{eq:rlt}  holds because both members are equal to $\int_0^t (\nabla f \, b)(s, X_s) \di s$. Thus it is natural to require the condition \eqref{eq:rlt}.
\item[(ii)]
In the case $B=C_T \bar {\mathcal C}_c^{(-\beta)+}$ we notice that the condition $ \nabla f b \in B$ is always satisfied. Indeed $\nabla f \in C_T \mathcal C^{\beta+}$ and $ b\in B $ 
 thus by \eqref{eq:bonyt} $\nabla f  b\in C_T \mathcal C^{(-\beta)+}$. Finally $\nabla f  b\in C_T \bar{ \mathcal C}_c^{(-\beta)+}$ because
we can construct the compactly supported sequence by considering $\nabla f  b_n$, where $ b_n $ is the compactly supported sequence that converges to $ b$ in $C_T \bar{ \mathcal C}_c^{(-\beta)+}$, using again \eqref{eq:bonyt}. Thus $ C^{1,2,B}_b$ reduces to 
\[
\{ f\in C_b^{1,2} \text{ such that } \dot f +\frac12 \Delta f  \in C_T \bar {\mathcal C}_c^{0+} \}
\]
and does not depend on $B$.
\end{itemize}
 \end{remark}

 Next we want to consider the case when $X$
 is an $\mathcal F^X$-Dirichlet process.
 In this case we  show that  the notion of solution of the  martingale problem with distributional drift is equivalent to the one of the reinforced $B$-solution. 
 Let us start with a remark.   
 
 \begin{remark}\label{rm:XX}
 If $X$ is a $B$-solution which is an $\mathcal F^X$-Dirichlet process, then $ [X, X]_t = t \text{I}_d$. Indeed, by Remark \ref{rm:cr} we have that $X_t = W_t^X + A^{X, B}_t(b)$ is the standard decomposition of the weak Dirichlet process $X$, and  by the uniqueness of the
weak Dirichlet decomposition  and the fact that $X$ is an $\mathcal F^X$-Dirichlet process  then $A^{X, B}_t(b)$
is a zero quadratic variation process and so $[X,X]_t = t \text{I}_d$.
 \end{remark}
\begin{proposition}\label{pr:Xfqvi}
Let  $B= C_T \bar{\mathcal C}_c^{(-\beta)+}$. If $(X, \mathbb P)$ satisfies
the martingale problem with distributional drift $b$
and  $X$ is an $\mathcal F^X$-Dirichlet process,
then $X$ is a reinforced $B$-solution according to Definition \ref{def:rlt}. 
\end{proposition}

\begin{proof}
First we notice that point (i) of Definition \ref{def:rlt} is satisfied by Theorem \ref{thm:cr}.
Next we check point (ii) and we write $A^X$ instead of $A^{X,B}$ for ease of notation.
Let $f\in C_b^{1,2,B}$. 
 Using the weak Dirichlet decomposition  since  $f\in C_b^{1,2}$  we have
\begin{align} \label{eq:fA}
\nonumber
  f(t, X_t) &=  M_t^f + \mathcal A^f_t\\
&= f(0, X_0) + \int_0^t (\nabla f)(r, X_r) \cdot \di W^X_r + \mathcal A^f_t,
\end{align}
having used Proposition \ref{pr:Mf}  to express the  martingale component part. 

On the other hand,
it easily follows that $f\in  \mathcal D_{\mathcal L}^B$ defined in \eqref{eq:DLB},
because $\mathcal L f = \nabla f b + g$,  where $g:= \dot f + \frac12 \Delta f \in C_T \bar{\mathcal C}_c^{0+} \subset B $ by assumption,  and $\nabla f b \in B$ as seen in Remark \ref{rm:B}, item (ii).
Since $X$ is a $B$-solution and an $\mathcal F^X$-Dirichlet process,
by Remark \ref{rm:XX} we have  $[X,X]_t = t \text{I}_d$. 
So 
by applying a slight adaptation of It\^o's formula \cite[Theorem 2.2]{russo_vallois95} to $f(t, X_t) $ for $f\in C_b^{1,2}$ we have
\begin{align}
  \nonumber f(t, X_t) =& f(0, X_0) + \int_0^t (\nabla f)(r, X_r)
                         \cdot \di W^X_r +  \int_0^t (\nabla f)(r, X_r) \cdot \di^- A^X_r (b)\\
\nonumber &+ \int_0^t (\partial_ t f+ \frac12 \Delta f)(r, X_r) \di r\\
\nonumber  =& f(0, X_0) + \int_0^t (\nabla f)(r, X_r) \cdot \di W^X_r +  \int_0^t (\nabla f)(r, X_r) \cdot \di^- A^X_r(b)\\
 &+A^X_t (\partial_ t f+ \frac12 \Delta f).
\label{eq:fAA}
\end{align}
We recall that $\partial_ t f+ \frac12 \Delta f \in C_T \bar{\mathcal C}_c^{0+}$ since $f \in C^{1,2,B}_b$, and so $A^X_t (\partial_ t f+ \frac12 \Delta f )$ is trivially well-defined. On the other hand $\partial_ t f+ \frac12 \Delta f = \mathcal L f - \nabla f \, b$, where $ \mathcal L f,\nabla f \, b \in B$ as noticed in Remark \ref{rm:B} item (ii), thus we can write
$$  A^X_t (\partial_ t f+ \frac12 \Delta f) =  A^X_t (\mathcal L f - \nabla f \, b ) = A^X_t (\mathcal L f )- A^X_t(  \nabla f \, b ).$$
Plugging this into
\eqref{eq:fAA}  and  comparing with \eqref{eq:fA} we get
\[
\mathcal A^f_t = \int_0^t (\nabla f)(r, X_r)  \cdot \di^- A^X_r \\
+A^X_t (\mathcal L f ) - A^X_t( \nabla f \, b ),
\]
hence applying \eqref{eq:AA} we conclude.
\end{proof}

The next result is the converse statement of Proposition \ref{pr:Xfqvi}. 
\begin{proposition}\label{pr:Xfqv}
  Let  $B= C_T \bar{\mathcal C}_c^{(-\beta)+}$  and $b \in B$. Let  $\mathbb P$ be a probability measure on  $(\Omega, \mathcal F)$. Let $X$ be a reinforced $B$-solution according to Definition \ref{def:rlt},
which is also an $\mathcal F^X$-Dirichlet process.
  Then $(X, \mathbb P)$ solves the martingale problem with distributional drift $b$. 
\end{proposition}

\begin{proof}
We need to show that for every $f\in \mathcal D_{\mathcal L}$ 
\[
f(t, X_t) -f(0, X_0) - \int_0^t (\mathcal Lf )(r, X_r) \di r
\]
is an $\mathcal F^X$-local martingale under $\mathbb P$. 
Since $f\in \mathcal D_{\mathcal L}$ we know that there exists  $l\in C_T \bar{\mathcal C}_c^{0+}$  such that $\mathcal L f = l$. 
By the density of $\mathcal S$ into $ \bar{\mathcal C}_c^{(-\beta)+}$, see \cite[Lemma 5.4]{issoglio_russoPDEa} and using \cite[Remark B.1]{issoglio_russoMK} we see that   $C_T\mathcal S$ is dense in $ C_T\bar{\mathcal C}_c^{(-\beta)+}$. Thus 
we can find a sequence  $(b_n) $ such that $ b_n \in C_T \bar{\mathcal C}_c^{0+}$  and $b_n \to b$ in $C_T {\mathcal C}^{(-\beta)+}$.
Let $\mathcal L_n u:= \partial_t u + \frac12 \Delta u + \nabla u \, b_n$ and let us consider the  PDE $\mathcal L_n f_n = l$ and $f_n(T) = f(T)$. By \cite[Remark 4.12]{issoglio_russoPDEa} we know that the unique solution $f_n\in C_T \mathcal C^{(1+\beta)+}$  is also the classical solution as given in \cite[Theorem 5.1.9]{lunardi95}, hence  $f_n\in  C^{1,2}_b  $. 
We recall that $X$ is a $B$-solution in the sense of Definition \ref{def:Bsol} and it is an $\mathcal F^X$-Dirichlet process with decomposition $X = W^X + A^{X,B}$ by Remark \ref{rm:XX}. By
 It\^o's formula \cite[Theorem 6.1]{rv4}, taking into account the  linearity of $A^{X,B}$ and the fact that $b_n \in C_T \bar{\mathcal C}_c^{0+}$, we have
\begin{align}
\label{eq:fnl}
\nonumber
f_n(t, X_t) =& f_n (0, X_0) + \int_0^t (\nabla f_n )(s, X_s) \cdot \di W_s
+ \int_0^t(\nabla f_n )(s, X_s)  \cdot \di^- A^{X,B}_s(b-b_n) \\ \nonumber
&+ \int_0^t (\nabla f_n )(s, X_s) b_n (s, X_s) \di s+ \frac12 \int_0^t (\Delta f_n)(s, X_s ) \di s + \int_0^t (\partial_s f_n)(s, X_s) \di s\\\nonumber
 =& f_n (0, X_0) + \int_0^t (\nabla f_n )(s, X_s) \cdot \di W_s + \int_0^t(\nabla f_n )(s, X_s) \cdot \di^- A^{X,B}_s(b-b_n)  \\ 
&+ \int_0^t l(s, X_s ) \di s ,
\end{align}
having used $\mathcal L_n f_n =l$  in  the last equality.
Using again the linearity of $A^{X,B}$ we have 
\begin{align}\label{eq:Abn}\nonumber
\int_0^t &(\nabla f_n )(s, X_s) \cdot \di^- A^{X,B}_s(b-b_n) \\
&=\int_0^t(\nabla f_n )(s, X_s) \cdot \di^- A^{X,B}_s(b) - \int_0^t(\nabla f_n )(s, X_s) \cdot \di^- A^{X,B}_s(b_n). 
\end{align}
The second integral on the RHS is  equal to $A^{X,B}_t (\nabla f_n \,b_n)$  by Remark \ref{rm:B} item (i) since $f_n \in C^{1,2}_b$ and $b_n \in C_T \bar  {\mathcal C}_c^{0+}$. 
Since $X$ is a reinforced $B$-solution, by \eqref{eq:rlt} the  first integral on the RHS  of \eqref{eq:Abn} gives $A^{X,B}_t (\nabla f_n \,b) $ so by additivity  we rewrite \eqref{eq:Abn} as
\begin{equation}\label{eq:Abni}
\int_0^t (\nabla f_n )(s, X_s) \cdot \di^- A^{X,B}_s(b-b_n) =
A^{X,B}_t (\nabla f_n \,b-\nabla f_n \,b_n).
\end{equation}
Plugging \eqref{eq:Abni} into \eqref{eq:fnl} we have 
\begin{align}\label{eq:fnm}
  f_n(t, X_t) &-f_n (0, X_0) - A^{X,B}_t (\nabla f_n \,b-\nabla f_n \,b_n) - \int_0^t l(s, X_s ) \di s  
  = \int_0^t (\nabla f_n )(s, X_s) \cdot \di W_s.
\end{align} 
Since $b_n \to b$ in $C_T \mathcal C^{-\beta}$ we then have  $f_n \to f$ in $C_T \mathcal C^{(1+\beta)+}$  and $ \nabla f_n \to \nabla f$ in   $C_T \mathcal C^{\beta+}$  by continuity results for PDE \eqref{eq:PDE}, see Section \ref{ssc:pointwise}. Thus the right-hand side of \eqref{eq:fnm} converges u.c.p.\ to $\int_0^t (\nabla f )(s, X_s) \cdot \di W_s$, which is a local martingale under  $\mathbb P$. Moreover the left-hand side of \eqref{eq:fnm} converges u.c.p. to $f(t, X_t) -f(0, X_0) - \int_0^t l(s, X_s ) \di s $ and since $l = \mathcal L f$ we conclude.  
\end{proof}

As a consequence we get a characterisation property for solutions of the SDE in terms of solutions to martingale problem. 

\begin{corollary}\label{cor:Xsoliff}
  Let  $B= C_T \bar{\mathcal C}_c^{(-\beta)+}$ and $b \in B$. Let  $\mathbb P$ be a probability measure on $(\Omega, \mathcal F)$. Suppose that $X$ is an $\mathcal F^X$-Dirichlet process. 
Then $X$ is a reinforced $B$-solution of the SDE
\[
X_t = X_0 + W_t + \int_0^t b(t, X_t) \di t
\]
if and only if $(X, \mathbb P)$ solves the martingale problem with distributional drift $b$ and initial condition $X_0 \sim \mu$.
\end{corollary}

\begin{proof}
  Combine Proposition \ref{pr:Xfqvi}  and Proposition
  \ref{pr:Xfqv}. 
\end{proof}

\appendix
\section{Some useful results from the literature}\label{app:lunardi}

In this Appendix we recall a useful theorem from  \cite{lunardi95} on existence and regularity results of parabolic PDEs. Before stating the theorem, we recall the notation used in the book, see \cite[Chapter 5]{lunardi95}.

The classical H\"older space  $\mathcal C^{2+\nu}(\R^d)$ for $0<\nu<1$ was introduced in Section \ref{sc:prelim}.  For functions of two variables $(t,x)\in[0,T]\times\R^d$  we consider the spaces introduced in  \cite[Section 5.1]{lunardi95}
\begin{align*}
C^{0,\nu}([0,T]\times \R^d) := \{ & f\in C([0,T]\times \R^d)  : f(t, \cdot)\in \mathcal  C^\nu(\R^d) \, \forall t\in[0,T], \sup_{t\in[0,T]} \|f(t, \cdot)\|_{\mathcal C^\nu}<\infty\},
\end{align*}
with norm
\[
\|f\|_{C^{0,\nu}([0,T]\times \R^d)} := \sup_{t\in[0,T]}  \|f(t, \cdot)\|_{\mathcal C^\nu}
\]
and 
\begin{align*}
C^{1,2+\nu}([0,T]\times \R^d) := \{ & f\in C^{1,2}([0,T]\times \R^d) :\partial_t f, \partial_{x_i, x_j} f \in C^{0,\nu} ([0,T]\times\R^d), \, \forall i, j=1, \ldots, d\}
\end{align*}
with norm 
\begin{align*}
\|f\|_{C^{1,2+\nu}([0,T]\times \R^d)} :=& \|f\|_\infty + \sum_{i=1}^d \|\partial_i f\|_\infty   + \|\partial_t f\|_\infty +  \sum_{i,j=1}^d \|\partial_{ij} f\|_{C^{0,\nu}([0,T]\times \R^d)}.
\end{align*}

\begin{remark}\label{rm:x1x2}
Note that   the following is an equivalent norm in  $\mathcal C^\nu$ 
\[
\sup_x |f(x)| +\sup_{x_1, x_2\, x_1\neq x_2} \frac{|f(x_1) - f(x_2)|}{|x_1 - x_2|^\nu},
\]
namely we can freely choose not to restrict $x_1, x_2$ to a bounded interval. 
\end{remark}

\begin{remark}\label{rm:C12buc}
Note that if $ f\in C_T\mathcal C^\nu$ then trivially we have $f\in C^{0,\nu}([0,T]\times \R^d)$ and  if $f\in C^{1,2+\nu}([0,T]\times \R^d)$ then trivially $f\in C^{1,2}_{buc}$.
\end{remark}

Let $a_{i,j}, b_i, c,f:[0,T]\times \R^d \to \R$ be uniformly continuous  and belonging to $C^{0,\nu}([0,T]\times\R^d)$ with $0<\nu<1$. Let $a$ satisfy the uniform ellipticity condition $\sum_{i,j=1}^d a_{i,j} (t,x) \xi_i \xi_j \geq \lambda |\xi|^2$ for $t\in[0,T],\,  x, \xi\in \R^d$, for some $\lambda>0$.
Let  $u_0\in \mathcal  C^{2+\nu}(\R^d)$. 

 Let us consider the second order operator
\[\mathcal A(t,x) = \sum_{i,j=1}^d a_{i,j}(t,x) \partial_{x_i x_j} + \sum_{i=1}^d b_{i}(t,x) \partial_{x_i} + c(t,x) \] 
and the PDE
\begin{equation}\label{eq:PDELunardi}
\left\{
\begin{array}{l}
\partial_t u(t,x) = \mathcal A(t,x) u(t,x) + f(t,x), \, (t,x) \in [0,T]\times\R^d\\
u(0,x) = u_0(x), \, x\in\R^d.
\end{array}
\right.
\end{equation}
\begin{theorem}[Theorem 5.1.9 in  \cite{lunardi95}]\label{thm:lunardi}
Let $a_{i,j}, b_i, c, f, u_0$ be as above.
Then PDE \eqref{eq:PDELunardi} has a unique solution $u\in C^{1,2+\nu}([0,T]\times\R^d )$ and
\[
\|u\|_{ C^{1,2+\nu}([0,T]\times\R^d )} \leq C ( \|u_0\|_{ \mathcal C^{2+\nu}(\R^d )} + \|f\|_{ C^{0,\nu}([0,T]\times\R^d )}),
\]
for some $C>0$.
\end{theorem}


{\bf ACKNOWLEDGEMENTS.}
We would like to thank the anonymous Referee for their careful reading, that has led to various improvements in the final version of the paper.  

{\bf FUNDING.}
The research of the first named author has been partially supported by the MIUR-PRIN 2022 project “Non-Markovian dynamics and non-local equations”, no.\ 202277N5H9.
The research of the second named author
has been partially supported by the  ANR-22-CE40-0015-01 project (SDAIM). 

{\bf CONFLICT OF INTEREST.}
All authors declare that they have no conflicts of interest.

{\bf DATA AVAILABILITY.}
No data are associated with this article.

\bibliographystyle{plain}

\bibliography{../../../../../BIBLIO_FILE/biblio}

\def\cprime{$'$}
\begin{thebibliography}{10}

\bibitem{athreya2020}
S.~Athreya, O.~Butkovsky, and L.~Mytnik.
\newblock Strong existence and uniqueness for stable stochastic differential
  equations with distributional drift.
\newblock {\em The Annals of Probability}, 48(1):178--210, 2020.

\bibitem{bahouri}
H.~Bahouri, J-Y. Chemin, and R.~Danchin.
\newblock {\em Fourier Analysis and Nonlinear Partial Differential Equations}.
\newblock Springer, 2011.

\bibitem{BarlowYor}
M.T. Barlow and M.~Yor.
\newblock Semi-martingale inequalities via the {G}arsia-{R}odemich-{R}umsey
  lemma, and applications to local times.
\newblock {\em Journal of Functional Analysis}, 49(2):198 -- 229, 1982.

\bibitem{bony}
J.-M. Bony.
\newblock Calcul symbolique et propagation des singularites pour les
  \'{e}quations aux d\'{e}riv\'{e}es partielles non lin\'{e}aires.
\newblock {\em Ann. Sci. Ec. Norm. Super.}, 14:209--246, 1981.

\bibitem{cannizzaro}
G.~{Cannizzaro} and K.~{Chouk}.
\newblock {Multidimensional SDEs with singular drift and universal construction
  of the polymer measure with white noise potential}.
\newblock {\em {Ann. Probab.}}, 46(3):1710--1763, 2018.

\bibitem{chaudru_menozzi}
P.-E. Chaudru~de Raynal and S.~Menozzi.
\newblock On {M}ultidimensional stable-driven {S}tochastic {D}ifferential
  {E}quations with {B}esov drift.
\newblock {\em Electronic Journal of Probability}, 27:1--52, 2022.

\bibitem{diel}
F.~Delarue and R.~Diel.
\newblock Rough paths and 1d {SDE} with a time dependent distributional drift:
  application to polymers.
\newblock {\em Probab. Theory Related Fields}, 165(1-2):1--63, 2016.

\bibitem{dellacherie}
C.~Dellacherie and P.-A. Meyer.
\newblock {\em Probabilities and potential. {Transl}. from the {French}},
  volume~29 of {\em North-Holland Math. Stud.}
\newblock Elsevier, Amsterdam, 1978.

\bibitem{dunford-schwartz}
N.~Dunford and J.~T. Schwartz.
\newblock {\em Linear operators. {P}art {I}}.
\newblock Wiley Classics Library. John Wiley \& Sons Inc., New York, 1988.
\newblock General theory, With the assistance of William G. Bade and Robert G.
  Bartle, Reprint of the 1958 original, A Wiley-Interscience Publication.

\bibitem{er}
M.~Errami and F.~Russo.
\newblock Covariation de convolutions de martingales.
\newblock {\em C. R. Acad. Sci. Paris S\'er. I Math.}, 326(5):601--606, 1998.

\bibitem{er2}
M.~Errami and F.~Russo.
\newblock {$n$}-covariation, generalized {D}irichlet processes and calculus
  with respect to finite cubic variation processes.
\newblock {\em Stochastic Process. Appl.}, 104(2):259--299, 2003.

\bibitem{flandoli_et.al14}
F.~Flandoli, E.~Issoglio, and F.~Russo.
\newblock Multidimensional {SDE}s with distributional coefficients.
\newblock {\em T. Am. Math. Soc.}, 369:1665--1688, 2017.

\bibitem{frw2}
F.~Flandoli, F.~Russo, and J.~Wolf.
\newblock Some {SDE}s with distributional drift. {I}. {G}eneral calculus.
\newblock {\em Osaka J. Math.}, 40(2):493--542, 2003.

\bibitem{frw1}
F.~Flandoli, F.~Russo, and J.~Wolf.
\newblock Some {SDE}s with distributional drift. {II}. {L}yons-{Z}heng
  structure, {I}t\^o's formula and semimartingale characterization.
\newblock {\em Random Oper. Stochastic Equations}, 12(2):145--184, 2004.

\bibitem{gozzi_russo06}
F.~Gozzi and F.~Russo.
\newblock Weak {D}irichlet processes with a stochastic control perspective.
\newblock {\em Stochastic Processes and their Applications}, 116(11):1563 --
  1583, 2006.

\bibitem{gubinelli_imkeller_perkowski}
M.~Gubinelli, P.~Imkeller, and N.~Perkowski.
\newblock Paracontrolled distributions and singular {PDE}s.
\newblock {\em Forum of Mathematics, Pi}, 3:75 pages, 2015.

\bibitem{iprt24}
E.~Issoglio, S.~Pagliarani, F.~Russo, and D.~Trevisani.
\newblock Degenerate {M}c{K}ean-{V}lasov equations with drift in anisotropic
  negative {B}esov spaces.
\newblock {\em Preprint Arxiv 2401.09165}, 2024.

\bibitem{issoglio_russoPDEa}
E.~Issoglio and F.~Russo.
\newblock A {PDE} with drift of negative {B}esov index and linear growth
  solutions.
\newblock {\em To appear: Differential and Integral Equations. Preprint Arxiv
  2212.04293}, 2022.

\bibitem{issoglio_russoMK}
E.~Issoglio and F.~Russo.
\newblock Mc{K}ean {SDE}s with singular coefficients.
\newblock {\em Annales de l'Institut Henri Poincar\'e. Probabilit\'es et
  Statistiques.}, 59(3):1530--1548, 2023.

\bibitem{karatzasShreve}
I.~Karatzas and S.E. Shreve.
\newblock {\em Brownian Motion and Stochastic Calculus}.
\newblock Graduate Texts in Mathematics. Springer New York, 1991.

\bibitem{lunardi95}
A.~Lunardi.
\newblock {\em Analytic semigroups and optimal regularity in parabolic
  problems}.
\newblock Progress in Nonlinear Differential Equations and their Applications,
  16. Birkh\"auser Verlag, Basel, 1995.

\bibitem{russo_trutnau07}
F.~Russo and G.~Trutnau.
\newblock Some parabolic {PDE}s whose drift is an irregular random noise in
  space.
\newblock {\em Ann. Probab.}, 35(6):2213--2262, 2007.

\bibitem{russo_vallois93}
F.~Russo and P.~Vallois.
\newblock Forward, backward and symmetric stochastic integration.
\newblock {\em Probab. Theory Related Fields}, 97(3):403--421, 1993.

\bibitem{russo_vallois95}
F.~Russo and P.~Vallois.
\newblock The generalized covariation process and {I}t\^o formula.
\newblock {\em Stochastic Processes and their Applications}, 59(1):81 -- 104,
  1995.

\bibitem{rv4}
F.~Russo and P.~Vallois.
\newblock Stochastic calculus with respect to continuous finite quadratic
  variation processes.
\newblock {\em Stochastics Stochastics Rep.}, 70(1-2):1--40, 2000.

\bibitem{Russo_Vallois_Book}
F.~Russo and P.~Vallois.
\newblock {\em Stochastic Calculus via Regularizations}, volume~11.
\newblock Springer International Publishing. Springer-Bocconi, 2022.

\bibitem{seignourel1998processus}
P.~Seignourel.
\newblock {\em Processus dans un milieu irr{\'e}gulier: une approche par les
  formes de Dirichlet}.
\newblock {\'E}cole polytechnique, 1998.

\bibitem{stroock_varadhan}
D.~W. Stroock and S.~R.~S. Varadhan.
\newblock {\em Multidimensional diffusion processes}, volume 233 of {\em
  Grundlehren der Mathematischen Wissenschaften [Fundamental Principles of
  Mathematical Sciences]}.
\newblock Springer-Verlag, Berlin, 1979.

\bibitem{ZhangZhao}
X.~Zhang and G.~Zhao.
\newblock Heat kernel and ergodicity of {SDE}s with distributional drifts.
\newblock {\em Preprint Arxiv 1710.10537}, 2017.

\bibitem{z}
A.~K. Zvonkin.
\newblock A transformation of the phase space of a diffusion process that
  removes the drift.
\newblock {\em Math. USSR, Sb.}, 22:129--149, 1975.

\end{thebibliography}

\end{document}